\documentclass[reqno]{amsart}
\usepackage{amsthm, amssymb, amsfonts}
\usepackage{lmodern}
\usepackage{color} 
\usepackage{hyperref}
\usepackage[T5]{fontenc}
\usepackage{amscd,amssymb}
\usepackage{mathrsfs}
\theoremstyle{plain}
\newtheorem{thm}{Theorem}[section]
\newtheorem{prop}[thm]{Proposition}
\newtheorem{lem}[thm]{Lemma}
\newtheorem{con}[thm]{Conjecture}
\newtheorem{corl}[thm]{Corollary}
\theoremstyle{definition}
\newtheorem{defn}[thm]{Definition}

\newtheorem{nota}[thm]{Notation}

\theoremstyle{plain}
\newtheorem{thms}{Theorem}[subsection]
\newtheorem{props}[thms]{Proposition}
\newtheorem{lems}[thms]{Lemma}
\newtheorem{corls}[thms]{Corollary}
\theoremstyle{definition}

\begin{document} 

\title[On Singer's conjecture for the fifth algebraic transfer]{On Singer's conjecture for the fifth\\ algebraic transfer} 

\author{Nguy\~\ecircumflex n Kh\'{\u a}c T\'in}

\footnotetext[1]{2000 {\it Mathematics Subject Classification}. Primary 55S10; 55S05, 55T15.}
\footnotetext[2]{{\it Keywords and phrases:} Steenrod squares, hit problem, algebraic transfer.}
\maketitle

\begin{abstract} Let $P_k:= \mathbb{F}_2[x_1,x_2,\ldots ,x_k]$  be the polynomial algebra in $k$ variables with the degree of each $x_i$ being $1,$ regarded as a module over the mod-$2$ Steenrod algebra $\mathcal{A},$ and let $GL_k$ be the general linear group over the prime field $\mathbb{F}_2$ which acts naturally on $P_k$. We study the {\it hit problem}, set up by Frank Peterson, of finding a minimal set of generators for the polynomial algebra $P_k$ as a module over the  mod-2 Steenrod algebra, $\mathcal{A}$. These results are used to study the Singer algebraic transfer which is a homomorphism from the homology of  the mod-$2$ Steenrod algebra, $\mbox{Tor}^{\mathcal{A}}_{k, k+n}(\mathbb{F}_2, \mathbb{F}_2),$ to the subspace of  $\mathbb{F}_2\otimes_{\mathcal{A}}P_k$ consisting of all the $GL_k$-invariant classes of degree $n.$

In this paper, we explicitly compute the hit problem for $k = 5$ and the degree $7.2^s-5$ with $s$ an arbitrary positive integer. Using this result, we show that Singer's conjecture for the algebraic transfer is true in the case $k=5$ and the above degree.
\end{abstract}

\medskip
\section{Introduction}\label{s1} 
\setcounter{equation}{0}

Let $E^k$ be an elementary abelian 2-group of rank $k$. Then, 
$$P_k:= H^*(E^k) \cong \mathbb F_2[x_1,x_2,\ldots ,x_k],$$ 
a polynomial algebra in  $k$ variables $x_1, x_2, \ldots , x_k$, each of degree 1. Here the cohomology is taken with coefficients in the prime field $\mathbb F_2$ of two elements. 

Being the cohomology of a group,  $P_k$ is a module over the mod-2 Steenrod algebra, $\mathcal A$.  The action of $\mathcal A$ on $P_k$ is determined by the elementary properties of the Steenrod squares $Sq^i$ and the Cartan formula (see Steenrod and Epstein~\cite{st}).

A polynomial $f$ in $P_k$  is called {\it hit} if it can be written as a finite sum $f = \sum_{u\geqslant 0}Sq^{2^u}(h_u)$ 
for suitable polynomials $h_u$.  That means $f$ belongs to  $\mathcal{A}^+P_k$,  where $\mathcal{A}^+$ denotes the augmentation ideal in $\mathcal{A}$.

Let $GL_k$ be the general linear group over the field $\mathbb F_2$. This group acts naturally on $P_k$ by matrix substitution. Since the two actions of $\mathcal A$ and $GL_k$ upon $P_k$ commute with each other, there is an action of $GL_k$ on $QP_k := \mathbb F_2\otimes_{\mathcal A}P_k$. 

Many authors  study the {\it hit problem} of determination of a minimal set  of generators for $P_k$ as a module over the Steenrod algebra, or equivalently, a basis of $QP_k$.   This problem has first been studied by Peterson \cite{pe}, Wood \cite{wo}, Singer \cite {si1}, Priddy \cite{pr},  who show its relationship to several classical problems in homotopy theory.

The vector space $QP_k$ was explicitly calculated by 
Peterson~\cite{pe} for $k=1, 2,$ by Kameko~\cite{ka} for $k=3$ and by Sum \cite{su3,su1} for $k = 4$. However, for $k > 4$, it is still open.

For a nonnegative integer $d$, denote by $(P_k)_d$ the subspace of $P_k$ consisting of all the homogeneous polynomials of degree $d$ in $P_k$ and by $(QP_k)_d$ the subspace of $QP_k$ consisting of all the classes represented by the elements in $(P_k)_d$. 
In \cite{si1}, Singer defined the algebraic transfer,  which is a homomorphism
$$\varphi_k :\text{Tor}^{\mathcal A}_{k,k+d} (\mathbb F_2,\mathbb F_2) \longrightarrow  (QP_k)_d^{GL_k}$$
from the homology of the Steenrod algebra to the subspace of $(QP_k)_d$ consisting of all the $GL_k$-invariant classes. It is a useful tool in describing the homology groups of the Steenrod algebra, $\text{Tor}^{\mathcal A}_{k,k+d} (\mathbb F_2,\mathbb F_2)$. The hit problem and the algebraic transfer was studied by many authors (see Boardman ~\cite{bo}, Bruner-H\`a-H\uhorn ng ~\cite{br}, Janfada \cite{jw1}, H\`a ~\cite{ha}, H\uhorn ng ~\cite{hu2, hu3}, Ch\ohorn n-H\`a~ \cite{cha1,cha2}, Minami ~\cite{mi}, Nam~ \cite{na2}, H\uhorn ng-Qu\`ynh~ \cite{hq}, Qu\`ynh~ \cite{qh}, Sum-T\'in~\cite{s-t15}, Sum~\cite{su5}, T\'in-Sum~\cite{t-s16} and others).

Singer showed in \cite{si1} that $\varphi_k$ is an isomorphism for $k=1,2$. Boardman showed in \cite{bo} that $\varphi_3$ is also an isomorphism. However, for any $k\geqslant 4$,  $\varphi_k$ is not a monomorphism in infinitely many degrees (see Singer \cite{si1}, H\uhorn ng \cite{hu3}.) Singer made the following conjecture.
\begin{con}[Singer \cite{si1}] The algebraic transfer $\varphi_k$ is an epimorphism for any $k \geqslant 0$.
\end{con}

 The conjecture is true for $k\leqslant 3$. Based on the results in \cite{su3,su1}, it can be verified for $k=4$. 

The purpose of the paper is to verify this conjecture for $k=5$ and the degree $7.2^s - 5$. The following is the main result of the paper.

\begin{thm}\label{mthm} Singer's conjecture is true for $k=5$ and the degree $7.2^{s} - 5$ with $s$ an arbitrary positive integer.
\end{thm}

To prove the theorem, we study the hit problem for $k = 5$ and the degree $7.2^s-5$. We have

\begin{thm}\label{pthm2} Let $m=7.2^s-5$ with $s$ a positive integer. Then

\medskip
{\rm i)} $\dim(QP_5)_{m} = 191$ for $s=1$,  and $\dim(QP_5)_{m} = 1245$ for any $s \geqslant 2$.

{\rm ii)} $(QP_5)_{m}^{GL_5} = 0$ for any $s \geqslant 1$.
\end{thm}
This theorem has been proved by Singer \cite{si1} for $s=1$. In \cite{hu3}, H\uhorn ng computed the dimensions of $(QP_5)_{m}$ and $(QP_5)_{m}^{GL_5}$ for $s = 2$ by using computer calculation. However, the detailed proof was unpublished at the time of the writing.

The proof of Theorem \ref{pthm2} is long and very technical.
One of our main tools is Kameko's homomorphism
$\widetilde{Sq}^0_*: QP_k \to QP_k,$ 
which is induced by an $\mathbb F_2$-linear map $\phi: P_k \to P_k$, given by
$$
\phi(x) = 
\begin{cases}y, &\text{if }x=x_1x_2\ldots x_ky^2,\\  
0, & \text{otherwise,} \end{cases}
$$
for any monomial $x \in P_k$. The map $\phi$ is not an $\mathcal A$-homomorphism. However, 
$\phi Sq^{2i} = Sq^{i}\phi$ and $\phi Sq^{2i+1} = 0$ for any non-negative integer $i$.

For a positive integer $n$, by $\mu(n)$ one means the smallest number $r$ for which it is possible to write $n = \sum_{1\leqslant i\leqslant r}(2^{u_i}-1),$ where $u_i >0$. 

\begin{thm}[see Kameko~\cite{ka}]\label{dlk} 
Let $d$ be a non-negative integer. If $\mu(2d+k)=k$, then 
$$(\widetilde{Sq}^0_*)_{(k,d)} := \widetilde{Sq}^0_*: (QP_k)_{2d+k}\longrightarrow (QP_k)_d$$
is an isomorphism of $GL_k$-modules. 
\end{thm}

Denote by  $\alpha(n)$ the number of ones in dyadic expansion of a positive integer $n$ and by $\zeta(n)$ the greatest integer $u$ such that $n$ is divisible by $2^u$. That means $n = 2^{\zeta(n)}m$ with $m$ an odd integer.  Set 
$$t(k,d) = \max\{0,k- \alpha(d+k) -\zeta(d+k)\}.$$ 
Sum proved in \cite{su2} the following.

\begin{thm}[see Sum \cite{su2}]\label{cthm} Let $d$ be an arbitrary non-negative integer. Then
$$(\widetilde{Sq}^0_*)^{s-t}: (QP_k)_{k(2^s-1) + 2^sd} \longrightarrow (QP_k)_{k(2^t-1) + 2^td}$$
is an isomorphism of $GL_k$-modules for every $s \geqslant t$ if and only if $t \geqslant  t(k,d)$.
\end{thm}

For $d=2$, we have $t(5,2) = 2$ and $5(2^s-1) + 2^sd = 7.2^s-5$. So, by Theorem \ref{cthm},
$(\widetilde{Sq}^0_*)^{s-2}: (QP_5)_{7.2^s-5} \longrightarrow (QP_5)_{23}$
is an isomorphism of $GL_5$-modules for every $s \geqslant 2$. Hence, we need only to prove Theorem \ref{pthm2} by computing $(QP_5)_{7.2^s-5}$ and $(QP_5)_{7.2^s-5}^{GL_5}$ for $s=1,2$.

From the results of Tangora \cite{ta}, Lin \cite{wl} and Chen \cite{che}, we obtain
$$\text{Tor}_{5,7.2^s}^{\mathcal A}(\mathbb F_2,\mathbb F_2) = \begin{cases}\langle (Ph_1)^*\rangle, &\text{if } s=1,\\ \langle (h_{s}g_{s -1})^*\rangle, &\text{if }s \geqslant 2,\end{cases}$$
and $h_{s}g_{s -1} \ne 0$, where $h_s$ denote the Adams element in $\text{\rm Ext}_{\mathcal A}^{1,2^s}(\mathbb F_2, \mathbb F_2)$, $P$ is the Adams periodicity operator in \cite{ad} and $g_{s-1} \in \text{\rm Ext}_{\mathcal A}^{4,2^{s+2}+2^{s+1}}(\mathbb F_2, \mathbb F_2)$ for $s \geqslant 2$. 
Hence, by Theorem \ref{pthm2}(ii), the homomorphism
$$\varphi_5: \text{Tor}_{5,7.2^s}^{\mathcal A}(\mathbb F_2,\mathbb F_2) \longrightarrow (QP_5)_{7.2^s-5}^{GL_5}$$
is an epimorphism. Theorem \ref{mthm} is proved.

Observe that in this case, $\varphi_5$ is not a monomorphism. So, our result confirms the one of H\uhorn ng.

\begin{corl}[See H\uhorn ng \cite{hu3}]\label{ch}  There are infinitely many degrees in which $\varphi_5$ is not a monomorphism.
\end{corl}
This paper is organized as follows. In Section \ref{s2}, we recall  some needed information on the admissible monomials in $P_k$, Singer's criterion on the hit monomials  and Kameko's homomorphism. Our results will be presented in Section \ref{s3}. Finally, in Section \ref{s7}, we list all the admissible monomials of degrees 10, 23 in $P_5$.

Theorems \ref{mthm}, \ref{pthm2} and \ref{cthm} have already been announced in \cite{t-s16}.

\section{Preliminaries}\label{s2}
\setcounter{equation}{0}

In this section, we recall some needed information from Kameko~\cite{ka} and Singer \cite{si2}, which will be used in the next section.
\begin{nota} We denote $\mathbb N_k = \{1,2, \ldots , k\}$ and
\begin{align*}
X_{\mathbb J} = X_{\{j_1,j_2,\ldots , j_s\}} =
 \prod_{j\in \mathbb N_k\setminus \mathbb J}x_j , \ \ \mathbb J = \{j_1,j_2,\ldots , j_s\}\subset \mathbb N_k,
\end{align*}
In particular, $X_{\mathbb N_k} =1,\
X_\emptyset = x_1x_2\ldots x_k,$ 
$X_j = x_1\ldots \hat x_j \ldots x_k, \ 1 \leqslant j \leqslant k,$ and $X:=X_k \in P_{k-1}.$

Let $\alpha_i(a)$ denote the $i$-th coefficient  in dyadic expansion of a non-negative integer $a$. That means
$a= \alpha_0(a)2^0+\alpha_1(a)2^1+\alpha_2(a)2^2+ \ldots ,$ for $ \alpha_i(a) =0$ or 1 with $i\geqslant 0$. 

Let $x=x_1^{a_1}x_2^{a_2}\ldots x_k^{a_k} \in P_k$. Denote $\nu_j(x) = a_j, 1 \leqslant j \leqslant k$.  
Set 
$$\mathbb J_t(x) = \{j \in \mathbb N_k :\alpha_t(\nu_j(x)) =0\},$$
for $t\geqslant 0$. Then, we have
$x = \prod_{t\geqslant 0}X_{\mathbb J_t(x)}^{2^t}.$ 
\end{nota}
\begin{defn}
For a monomial  $x$ in $P_k$,  define two sequences associated with $x$ by
\begin{align*} 
\omega(x)=(\omega_1(x),\omega_2(x),\ldots , \omega_i(x), \ldots),\ \
\sigma(x) = (\nu_1(x),\nu_2(x),\ldots ,\nu_k(x)),
\end{align*} 
where
$\omega_i(x) = \sum_{1\leqslant j \leqslant k} \alpha_{i-1}(\nu_j(x))= \deg X_{\mathbb J_{i-1}(x)},\ i \geqslant 1.$
The sequence $\omega(x)$ is called  the weight vector of $x$. 

Let $\omega=(\omega_1,\omega_2,\ldots , \omega_i, \ldots)$ be a sequence of non-negative integers.  The sequence $\omega$ is called  the weight vector if $\omega_i = 0$ for $i \gg 0$.
\end{defn}

The sets of all the weight vectors and the exponent vectors are given the left lexicographical order. 

  For a  weight vector $\omega$,  we define $\deg \omega = \sum_{i > 0}2^{i-1}\omega_i$.  
Denote by   $P_k(\omega)$ the subspace of $P_k$ spanned by all monomials $y$ such that
$\deg y = \deg \omega$, $\omega(y) \leqslant \omega$, and by $P_k^-(\omega)$ the subspace of $P_k$ spanned by all monomials $y \in P_k(\omega)$  such that $\omega(y) < \omega$. 

\begin{defn}\label{dfn2} Let $\omega$ be a weight vector and $f, g$ two polynomials  of the same degree in $P_k$. 

i) $f \equiv g$ if and only if $f - g \in \mathcal A^+P_k$. If $f \equiv 0$ then $f$ is called {\it hit}.

ii) $f \equiv_{\omega} g$ if and only if $f - g \in \mathcal A^+P_k+P_k^-(\omega)$. 
\end{defn}

Obviously, the relations $\equiv$ and $\equiv_{\omega}$ are equivalence ones. Denote by $QP_k(\omega)$ the quotient of $P_k(\omega)$ by the equivalence relation $\equiv_\omega$. Then, we have 
$$QP_k(\omega)= P_k(\omega)/ ((\mathcal A^+P_k\cap P_k(\omega))+P_k^-(\omega)).$$  

For a  polynomial $f \in  P_k$, we denote by $[f]$ the class in $QP_k$ represented by $f$. If  $\omega$ is a weight vector, then denote by $[f]_\omega$ the class represented by $f$. Denote by $|S|$ the cardinal of a set $S$.

It is easy to see that
$$ QP_k(\omega) \cong QP_k^\omega := \langle\{[x] \in QP_k : x \text{ \rm  is admissible and } \omega(x) = \omega\}\rangle. $$
So, we get
$$(QP_k)_n = \bigoplus_{\deg \omega = n}QP_k^\omega \cong \bigoplus_{\deg \omega = n}QP_k(\omega).$$
Hence, we can identify the vector space $QP_k(\omega)$ with $QP_k^\omega \subset QP_k$. 

For $1 \leqslant i \leqslant k$, define the $\mathcal{A}$-homomorphism  $g_i:P_k \to P_k$, which is determined by $g_i(x_i) = x_{i+1}, g_i(x_{i+1}) = x_i$, $g_i(x_j) = x_j$ for $j \ne i, i+1,\ 1 \leqslant i < k$, and $g_k(x_1) = x_1+x_2$,  $g_k(x_j) = x_j$ for $j > 1$.  Note that the general linear group $GL_k$ is generated by the matrices associated with $g_i,\ 1\leqslant i \leqslant k,$ and the symmetric group $\Sigma_k$ is generated by the ones associated with $g_i,\ 1 \leqslant i < k$. 
So, a homogeneous polynomial $f \in P_k$ is an $GL_k$-invariant if and only if $g_i(f) \equiv f$ for $1 \leqslant i\leqslant k$.  If $g_i(f) \equiv f$ for $1 \leqslant i < k$, then $f$ is an $\Sigma_k$-invariant.

We note that the weight vector of a monomial is invariant under the permutation of the generators $x_i$, hence $QP_k(\omega)$ has an action of the symmetric group $\Sigma_k$. 
 Furthermore, we have the following.

\begin{lem}[See Sum \cite{su2}]\label{bdm} Let $\omega$ be a weight vector. Then, $QP_k(\omega)$ is the $GL_k$-module.
\end{lem}

\begin{defn}\label{defn3} 
Let $x, y$ be monomials of the same degree in $P_k$. We say that $x <y$ if and only if one of the following holds:  

i) $\omega (x) < \omega(y)$;

ii) $\omega (x) = \omega(y)$ and $\sigma(x) < \sigma(y).$
\end{defn}

\begin{defn}
A monomial $x$ is said to be inadmissible if there exist monomials $y_1,y_2,\ldots, y_m$ such that $y_t<x$ for $t=1,2,\ldots , m$ and $x - \sum_{t=1}^my_t \in \mathcal A^+P_k.$ 

A monomial $x$ is said to be admissible if it is not inadmissible.
\end{defn}

Obviously, the set of all the admissible monomials of degree $n$ in $P_k$ is a minimal set of $\mathcal{A}$-generators for $P_k$ in degree $n$. 

\begin{thm}[See Kameko \cite{ka}]\label{dlcb1}  
 Let $x, y, w$ be monomials in $P_k$ such that $\omega_i(x) = 0$ for $i > r>0$, $\omega_s(w) \ne 0$ and   $\omega_i(w) = 0$ for $i > s>0$.

{\rm i)}  If  $w$ is inadmissible, then  $xw^{2^r}$ is also inadmissible.

{\rm ii)}  If $w$ is strictly inadmissible, then $wy^{2^{s}}$ is also strictly inadmissible.
\end{thm} 

Now, we recall a result of Singer \cite{si2} on the hit monomials in $P_k$. 

\begin{defn}\label{spi}  A monomial $z$ in $P_k$   is called a spike if $\nu_j(z)=2^{t_j}-1$ for $t_j$ a non-negative integer and $j=1,2, \ldots , k$. If $z$ is a spike with $t_1>t_2>\ldots >t_{r-1}\geqslant t_r>0$ and $t_j=0$ for $j>r,$ then it is called the minimal spike.
\end{defn}

In \cite{si2}, Singer showed that if $\mu(n) \leqslant k$, then there exists uniquely a minimal spike of degree $n$ in $P_k$. 

\begin{lem}[See \cite{sp}]\label{bdbs}
 All the spikes in $P_k$ are admissible and their weight vectors are weakly decreasing. 
Furthermore, if a weight vector $\omega$ is weakly decreasing and $\omega_1 \leqslant k$, then there is a spike $z$ in $P_k$ such that $\omega (z) = \omega$. 
\end{lem}

The following is a criterion for the hit monomials in $P_k$.

\begin{thm}[See Singer~\cite{si2}]\label{dlsig} Suppose $x \in P_k$ is a monomial of degree $n$, where $\mu(n) \leqslant k$. Let $z$ be the minimal spike of degree $n$. If $\omega(x) < \omega(z)$, then $x$ is hit.
\end{thm}

This result implies a result of Wood, which original is a conjecture of Peterson~\cite{pe}.
 
\begin{thm}[See Wood~\cite{wo}]\label{dlmd1} 
If $\mu(n) > k$, then $(QP_k)_n = 0$.
\end{thm} 

Now, we recall some notations and definitions in \cite{su1}, which will be used in the next sections. We set 
\begin{align*} 
P_k^0 &=\langle\{x=x_1^{a_1}x_2^{a_2}\ldots x_k^{a_k} \ : \ a_1a_2\ldots a_k=0\}\rangle,
\\ P_k^+ &= \langle\{x=x_1^{a_1}x_2^{a_2}\ldots x_k^{a_k} \ : \ a_1a_2\ldots a_k>0\}\rangle. 
\end{align*}

It is easy to see that $P_k^0$ and $P_k^+$ are the $\mathcal{A}$-submodules of $P_k$. Furthermore, we have the following.

\begin{prop}\label{2.7} We have a direct summand decomposition of the $\mathbb F_2$-vector spaces
$QP_k =QP_k^0 \oplus  QP_k^+.$
Here $QP_k^0 = \mathbb F_2\otimes_{\mathcal A}P_k^0$ and  $QP_k^+ = \mathbb F_2\otimes_{\mathcal A}P_k^+$.
\end{prop}

\begin{defn}
For any $1 \leqslant i < j \leqslant k$, we define the homomorphism $p_{(i;j)}: P_k \to P_{k-1}$ of algebras by substituting
$$p_{(i;j)}(x_u) =\begin{cases} x_u, &\text{ if } 1 \leqslant u < i,\\
x_{j-1}, &\text{ if }  u = i,\\  
x_{u-1},&\text{ if } i< u \leqslant k.
\end{cases}$$
For $1 \leqslant i \leqslant k$, define the homomorphism $f_i: P_{k-1} \to P_k$ of algebras by substituting
$$f_i(x_j) = \begin{cases} x_j, &\text{ if } 1 \leqslant j <i,\\
x_{j+1}, &\text{ if } i \leqslant j <k.
\end{cases}$$
Then, $p_{(i;j)}$ is a homomorphism of $\mathcal A$-modules
and $p_{(i;j)}(f_i(y)) = y$ for any $y \in P_{k-1}$. 
\end{defn}

For a subset $B \subset P_k,$ we denote $[B] = \{[f]:\,f\in  B \}.$ If $B \subset P_k(\omega),$ then we set $[B]_\omega = \{[f]_\omega:\, f\in  B \}.$ From Theorem \ref{dlsig}, we see that if $\omega$ is the weight vector of a minimal spike in $P_k$, then $[B]_\omega = [B].$
Obviously, we have
\begin{prop}\label{mdp0}
 It is easy to see that if $B$ is a minimal set of generators for $\mathcal A$-module $P_{k-1}$ in degree $n$, then  $f(B)= \bigcup _{i=1}^kf_i(B)$  is a minimal set of generators for $\mathcal A$-module $P_k^0$ in degree $n$.
\end{prop}

From now on, we denote by $B_{k}(n)$ the set of all admissible monomials of degree $n$  in $P_k$, $B_{k}^0(n) = B_{k}(n)\cap P_k^0$, $B_{k}^+(n) = B_{k}(n)\cap P_k^+$. For a weight vector $\omega$ of degree $n$, we set $B_k(\omega) = B_{k}(n)\cap P_k(\omega)$, $B_k^+(\omega) = B_{k}^+(n)\cap P_k(\omega)$. 

Then, $[B_k(\omega)]_\omega$ and $[B_k^+(\omega)]_\omega$, are respectively the basses of the $\mathbb F_2$-vector spaces $QP_k(\omega)$ and $QP_k^+(\omega) := QP_k(\omega)\cap QP_k^+$.

 For any monomials $z, z_1, z_2, \ldots, z_m$ in $P_k(\omega)$ with $m\geqslant 1$, we denote 
\begin{align*}
\Sigma_k(z_1, z_2, \ldots, z_m) &= \{\sigma z_t : \sigma \in \Sigma_k, 1\leqslant t \leqslant m\} \subset P_k(\omega),\\
[B(z_1, z_2, \ldots, z_m)]_\omega &= [B_k(\omega)]_\omega\cap \langle [\Sigma_k(z_1, z_2, \ldots, z_m)]_\omega\rangle,\\ 
p(z) &= \sum_{y \in B_k(n)\cap\Sigma_k(z)}y.
\end{align*}
Obviously, $ \langle [\Sigma_k(z_1, z_2, \ldots, z_m)]_\omega\rangle$ is an $\Sigma_k$-submodule of $QP_k(\omega)$. Furthermore, it is the $\Sigma_k$-module generated by the set $\{[z_1]_\omega,[z_2]_\omega,\ldots , [z_m]_\omega\}$.

\section{Proofs of the results}\label{s3}

In this section we prove Theorem \ref{pthm2} by explicitly determining all admissible monomials of degree $7.2^s-5$ in $P_5$. Using this results, we determine the space $(QP_5)_{7.2^s-5}^{GL_5}$ for all $s \geqslant 1$. Recall that by Theorem \ref{cthm},
$$(\widetilde{Sq}^0_*)^{s-2}: (QP_5)_{7.2^s-5} \longrightarrow (QP_5)_{23}$$
is an isomorphism of $GL_5$-modules for every $s \geqslant 2$. So, we need only to prove the theorem for $s=1,2$.

\subsection{The case $s=1$}\

\medskip
For $s=1$, we have $7.2^s-5 = 9$. Since Kameko's homomorphism 
$$(\widetilde{Sq}^0_*)_{(5,2)}: (QP_5)_{9} \longrightarrow (QP_5)_{2}$$
is an epimorphism, we have 
$(QP_5)_{9} \cong \text{Ker}(\widetilde{Sq}^0_*)_{(5,2)} \bigoplus (QP_5)_{2}$. It is easy to see that
$B_5(2) =\{x_ix_j  : 1 \leqslant i < j \leqslant 5\}$. So, we need only to compute $\text{Ker}(\widetilde{Sq}^0_*)_{(5,2)}$.

\begin{lems}\label{bdday1}
If $x$  is an admissible monomial of degree $9$ and $[x] \in \text{\rm Ker}(\widetilde{Sq}^0_*)_{(5,2)}$, then either
$\omega(x) = (3,1,1)$ or $\omega(x) = (3,3).$
\end{lems}
\begin{proof}  Observe that $z = x_1^{7}x_2x_3$ is the minimal spike of degree $9$ in $P_5$ and  $\omega(z) =  (3,1,1)$. Since $[x] \ne 0$, by Theorem \ref{dlsig}, either $\omega_1(x) =3$ or $\omega_1(x) =5$. If $\omega_1(x) =5$, then $x = X_\emptyset y^2$ with $y$ a monomial of degree $2$ in $P_5$. Since $x$ is admissible, by Theorem \ref{dlcb1}, $y$ is admissible. So, $(\widetilde{Sq}^0_*)_{(5,2)}([x]) = [y] \ne 0.$ This contradicts the fact that $[x] \in \text{Ker}(\widetilde{Sq}^0_*)_{(5,2)}$, hence $\omega_1(x) =3$. Then, we have $x = x_ix_jx_\ell y_1^2$ with $y_1$ an admissible monomial of degree $3$ in $P_5$. It is easy to see that either $\omega(y_1) =(1,1)$ or $\omega(y_1) =(3,0)$. The lemma is proved.
\end{proof}

Using this lemma, we see that
$$\text{\rm Ker}(\widetilde{Sq}^0_*)_{(5,2)} = (QP_5^0)_9 \bigoplus QP_5^+(3,1,1) \bigoplus QP_5^+(3,3).$$
From a result in \cite{su5}, we have $\dim(QP_5^0)_9 = 160$.
\begin{props}\label{mds1} $B_5^+(3,1,1)$ is the set of the monomials $x_1x_j^2x_\ell^4x_ux_v$ such that $(1,j,\ell,u,v)$ is a permutation of $(1,2,3,4,5)$, $j<\ell$ and $u < v$.
\end{props}
\begin{proof} Let $x$ be an admissible monomial in $P_5^+$ and $\omega(x) = (3,1,1)$. Then $x = x_ix_j^2x_\ell^4x_ux_v$ with $(i,j,\ell,u,v)$ a permutation of $(1,2,3,4,5)$ and $i < u < v$. If $\ell < j$, then using the Cartan formula, we have
$$ x = x_ix_j^4x_\ell^2x_ux_v + Sq^2(x_ix_j^2x_\ell^2x_ux_v) \ \text{ mod}(P_5^-(3,1,1)).$$
So, $x$ is inadmissible, hence $j< \ell$. If $j =1$, then
$$ x = x_i^2x_jx_\ell^4x_ux_v + x_ix_jx_\ell^4x_u^2x_v +x_ix_jx_\ell^4x_ux_v^2+ Sq^1(x_ix_jx_\ell^4x_ux_v).$$
Hence, $x$ is inadmissible, so $1 < j$. Since $i<u<v$ and $j < \ell$, we obtain $i=1$.

By a direct computation, we see that the set 
$$\{[x_1x_j^2x_\ell^4x_ux_v]: 1<j<\ell, 1 < u< v\}$$ 
is linearly independent in $QP_5.$ The proposition follows.
\end{proof}
By a similar argument as given in the proof of Proposition \ref{mds1}, we get the following.
\begin{props}\label{mds2} $B_5^+(3,3)$ is the set of the monomials $x_1x_j^2x_\ell^2x_ux_v^3$ such that $(1,j,\ell,u,v)$ is a permutation of $(1,2,3,4,5)$, and $j<\ell$.  
\end{props}

It is easy to see that $|B_5(2)| = 10$, $|B_5^+(3,1,1)| = 6$ and $|B_5^+(3,3)| = 15$. Hence, the first part of Theorem \ref{pthm2} is proved for $s=1$. 

By an easy computation, one gets
\begin{lems}\label{bdnn} $(QP_5)_2^{GL_5} = 0$.
\end{lems}

Since Kameko's homomorphism is an epimorphism of $GL_5$-modules, Lemma \ref{bdnn} implies $(QP_5)_{9}^{GL_5}= \text{\rm Ker}(\widetilde{Sq}_*^0)_{(5,2)}^{GL_5}.$

Using the above results, we see that $\dim ({\rm Ker}(\widetilde{Sq}_*^0)_{(5,2)}) = 181$ with the basis $\bigcup_{i=1}^6[B_5(u_i)]$, where 
\begin{align*} &u_{1} =  x_1x_2x_3^{7},\ u_{2} =  x_1^{3}x_2^{3}x_3^{3},\ u_{3} =  x_1x_2^{3}x_3^{5},\\ 
&u_{4} =  x_1x_2x_3^{2}x_4^{2}x_5^{3},\ u_{5} =  x_1x_2x_3^{2}x_4^{5},\ u_{6} =  x_1x_2^{2}x_3^{3}x_4^{3}.
\end{align*} 

By a direct computation, we obtain the following lemma.
\begin{lems}\label{bdss1}\ 
 We have a direct summand decomposition of the $\Sigma_5$-modules:
$$({\rm \rm Ker}(\widetilde{Sq}_*^0)_{(5,2)} = \bigoplus_{i=1}^4\langle[\Sigma_5(u_i)]\rangle\bigoplus \langle[\Sigma_5(u_5,u_6)]\rangle.$$
\end{lems}
\begin{lems}\label{bdss2}\

{\rm i)} $\langle[\Sigma_5(u_i)]\rangle ^{\Sigma_5}= \langle [p(u_i)]\rangle,\ i= 1,2,3$.

{\rm ii)} $\langle[\Sigma_5(u_4)]\rangle^{\Sigma_5}= \langle [p_1], [p_2]\rangle$, where
\begin{align*}
p_1 &= x_1x_2x_3x_4^{2}x_5^{4} + x_1x_2x_3^{2}x_4x_5^{4} + x_1x_2x_3^{2}x_4^{4}x_5 \\
&\quad+ x_1x_2^{2}x_3x_4x_5^{4} + x_1x_2^{2}x_3x_4^{4}x_5 + x_1x_2^{2}x_3^{4}x_4x_5,\\
p_2 &= x_1x_2x_3x_4^{2}x_5^{4} + x_1x_2x_3^{2}x_4x_5^{4} + x_1x_2x_3^{2}x_4^{4}x_5 + x_1x_2x_3^{2}x_4^{2}x_5^{3} + x_1x_2x_3^{2}x_4^{3}x_5^{2}\\
&\quad + x_1x_2x_3^{3}x_4^{2}x_5^{2} + x_1x_2^{2}x_3x_4^{2}x_5^{3} + x_1x_2^{2}x_3x_4^{3}x_5^{2} + x_1x_2^{2}x_3^{2}x_4x_5^{3} + x_1x_2^{2}x_3^{2}x_4^{3}x_5\\
&\quad + x_1x_2^{2}x_3^{3}x_4x_5^{2} + x_1x_2^{2}x_3^{3}x_4^{2}x_5 + x_1x_2^{3}x_3x_4^{2}x_5^{2} + x_1x_2^{3}x_3^{2}x_4x_5^{2} + x_1x_2^{3}x_3^{2}x_4^{2}x_5\\
&\quad + x_1^{3}x_2x_3x_4^{2}x_5^{2} + x_1^{3}x_2x_3^{2}x_4x_5^{2} + x_1^{3}x_2x_3^{2}x_4^{2}x_5.
\end{align*}

{\rm iii)} $ \langle[\Sigma_5(u_5,u_6)]\rangle^{\Sigma_5} = \langle [p_3]\rangle$, where
\begin{align*}
p_3 &= x_2x_3x_4x_5^{6} + x_2x_3x_4^{6}x_5 + x_2x_3^{6}x_4x_5 + x_1x_3x_4x_5^{6} + x_1x_3x_4^{6}x_5 + x_1x_3^{6}x_4x_5\\
&\quad + x_1x_2x_4x_5^{6} + x_1x_2x_4^{6}x_5 + x_1x_2x_3x_5^{6} + x_1x_2x_3x_4^{6} + x_1x_2x_3^{6}x_5 + x_1x_2x_3^{6}x_4\\
&\quad + x_1x_2^{6}x_4x_5 + x_1x_2^{6}x_3x_5 + x_1x_2^{6}x_3x_4 + x_2^{3}x_3x_4x_5^{4} + x_2^{3}x_3x_4^{4}x_5 + x_2^{3}x_3^{4}x_4x_5\\
&\quad + x_1^{3}x_3x_4x_5^{4} + x_1^{3}x_3x_4^{4}x_5 + x_1^{3}x_3^{4}x_4x_5 + x_1^{3}x_2x_4x_5^{4} + x_1^{3}x_2x_4^{4}x_5 + x_1^{3}x_2x_3x_5^{4}\\
&\quad + x_1^{3}x_2x_3x_4^{4} + x_1^{3}x_2x_3^{4}x_5 + x_1^{3}x_2x_3^{4}x_4 + x_1^{3}x_2^{4}x_4x_5 + x_1^{3}x_2^{4}x_3x_5 + x_1^{3}x_2^{4}x_3x_4.
\end{align*}
\end{lems}
\begin{proof} We prove the Part (ii) of the lemma. The others can be proved by a similar computation. From Propositions \ref{mds1} and \ref{mds2}, we see that $\dim\langle[\Sigma_5(u_4)]\rangle =21$ with a basis consisting of the classes 
represented by the following monomials:
\begin{align*}&v_{1} =  x_1x_2x_3x_4^{2}x_5^{4},\  v_{2} =  x_1x_2x_3^{2}x_4x_5^{4},\  v_{3} =  x_1x_2x_3^{2}x_4^{4}x_5,\  v_{4} =  x_1x_2^{2}x_3x_4x_5^{4}, \\ & v_{5} =  x_1x_2^{2}x_3x_4^{4}x_5,\  v_{6} =  x_1x_2^{2}x_3^{4}x_4x_5,\  v_{7} =  x_1x_2x_3^{2}x_4^{2}x_5^{3},\  v_{8} =  x_1x_2x_3^{2}x_4^{3}x_5^{2}, \\ & v_{9} =  x_1x_2x_3^{3}x_4^{2}x_5^{2},\  v_{10} =  x_1x_2^{2}x_3x_4^{2}x_5^{3},\  v_{11} =  x_1x_2^{2}x_3x_4^{3}x_5^{2},\  v_{12} =  x_1x_2^{2}x_3^{2}x_4x_5^{3}, \\ & v_{13} =  x_1x_2^{2}x_3^{2}x_4^{3}x_5,\  v_{14} =  x_1x_2^{2}x_3^{3}x_4x_5^{2},\  v_{15} =  x_1x_2^{2}x_3^{3}x_4^{2}x_5,\  v_{16} =  x_1x_2^{3}x_3x_4^{2}x_5^{2}, \\ & v_{17} =  x_1x_2^{3}x_3^{2}x_4x_5^{2},\  v_{18} =  x_1x_2^{3}x_3^{2}x_4^{2}x_5,\  v_{19} =  x_1^{3}x_2x_3x_4^{2}x_5^{2},\  v_{20} =  x_1^{3}x_2x_3^{2}x_4x_5^{2}, \\ & v_{21} =  x_1^{3}x_2x_3^{2}x_4^{2}x_5. 
\end{align*}
Suppose $f = \sum_{t=1}^{21}\gamma_tv_t$ with $\gamma_t \in \mathbb F_2$ and $[f] \in \langle[\Sigma_5(u_4)]\rangle^{\Sigma_5}$. By a direct computation, we have
\begin{align*} g_1(f) + f &\equiv (\gamma_{4}  + \gamma_{5} + \gamma_{10}  + \gamma_{11}) v_{1} +  (\gamma_{4}  + \gamma_{6} + \gamma_{12}  + \gamma_{14}) v_{2}\\
&\quad +   (\gamma_{5}  + \gamma_{6} + \gamma_{13}  + \gamma_{15}) v_{3} +  (\gamma_{10}  + \gamma_{12}) v_{7} +  (\gamma_{11}  + \gamma_{13}) v_{8}\\
&\quad +  (\gamma_{14}  + \gamma_{15}) v_{9} +  (\gamma_{16}  + \gamma_{19}) v_{16} +  (\gamma_{17}  + \gamma_{20}) v_{17} +  (\gamma_{18}  + \gamma_{21}) v_{18}\\
&\quad +  (\gamma_{16}  + \gamma_{19}) v_{19} +  (\gamma_{17}  + \gamma_{20}) v_{20} +  (\gamma_{18}  + \gamma_{21}) v_{21} \equiv 0,\\
\end{align*}
This relation implies
\begin{align}\label{ct1} \begin{cases}\gamma_{4}  + \gamma_{5} + \gamma_{10}  + \gamma_{11}= \gamma_{4}  + \gamma_{6} + \gamma_{12}  + \gamma_{14} = \gamma_{5}  + \gamma_{6} + \gamma_{13}  + \gamma_{15} =0\\
\gamma_{10} = \gamma_{12},\ \gamma_{11} = \gamma_{13},\ \gamma_{14} = \gamma_{15},\ \gamma_{16} = \gamma_{19},\ \gamma_{17}  = \gamma_{20},\ \gamma_{18}  = \gamma_{21}. \end{cases}
\end{align}
With the aid of (\ref{ct1}), we have
\begin{align*} g_2(f) + f &\equiv 
(\gamma_{17}  + \gamma_{18}) v_{1} + 
(\gamma_{2}  + \gamma_{4} + \gamma_{17}) v_{2} + 
(\gamma_{3}  + \gamma_{5} + \gamma_{18}) v_{3} \\ &\quad + 
(\gamma_{2}  + \gamma_{4} + \gamma_{17}) v_{4} + 
(\gamma_{3}  + \gamma_{5} + \gamma_{18}) v_{5} + 
(\gamma_{7}  + \gamma_{10}) v_{7}\\ &\quad + 
(\gamma_{8}  + \gamma_{11}) v_{8} + 
(\gamma_{9}  + \gamma_{16}) v_{9} + 
(\gamma_{7}  + \gamma_{10}) v_{10} + 
(\gamma_{8}  + \gamma_{11}) v_{11}\\ &\quad + 
(\gamma_{14}  + \gamma_{17}) v_{14} + 
(\gamma_{14}  + \gamma_{18}) v_{15} + 
(\gamma_{9}  + \gamma_{16}) v_{16}\\ &\quad + 
(\gamma_{14}  + \gamma_{17}) v_{17} + 
(\gamma_{14}  + \gamma_{18}) v_{18} + 
(\gamma_{17}  + \gamma_{18}) v_{19} \equiv 0.
\end{align*}
From the last equality, we get 
\begin{align}\label{ct2} \begin{cases}
\gamma_{2}  + \gamma_{4} + \gamma_{14}= 
\gamma_{3}  + \gamma_{5} + \gamma_{14} = 0,\\
\gamma_{2}  + \gamma_{4} + \gamma_{14}= 
\gamma_{3}  + \gamma_{5} + \gamma_{14} = 0 ,\\
\gamma_{14} = \gamma_{18},\ 
\gamma_{7} = \gamma_{10},\  
\gamma_{8}  = \gamma_{11},\ 
\gamma_{9}  = \gamma_{16},\ 
\gamma_{14}  = \gamma_{17}. 
\end{cases}
\end{align}
By a direct computation using (\ref{ct1}) and (\ref{ct2}), we obtain
\begin{align*} g_3(f) + f &\equiv 
(\gamma_{1}  + \gamma_{2}) v_{1} + 
(\gamma_{1}  + \gamma_{2}) v_{2} + 
(\gamma_{5}  + \gamma_{6}) v_{5} + 
(\gamma_{5}  + \gamma_{6}) v_{6}\\ &\quad + 
(\gamma_{8}  + \gamma_{9}) v_{8} + 
(\gamma_{8}  + \gamma_{9}) v_{9} + 
(\gamma_{8}  + \gamma_{14}) v_{11} + 
(\gamma_{8}  + \gamma_{14}) v_{13}\\ &\quad + 
(\gamma_{8}  + \gamma_{14}) v_{14} + 
(\gamma_{8}  + \gamma_{14}) v_{15} + 
(\gamma_{9}  + \gamma_{14}) v_{16}\\ &\quad + 
(\gamma_{9}  + \gamma_{14}) v_{17} + 
(\gamma_{9}  + \gamma_{14}) v_{19} + 
(\gamma_{9}  + \gamma_{14}) v_{20} \equiv 0. 
\end{align*}
This implies
\begin{align}\label{ct3} 
\gamma_{1}  + \gamma_{2}= 
\gamma_{5}  + \gamma_{6} = 
\gamma_{8}  + \gamma_{9} = 
\gamma_{8}  + \gamma_{14} = 0. 
\end{align}
By using (\ref{ct1}), (\ref{ct2}) and (\ref{ct3}), we have
\begin{align*} g_4(f) + f &\equiv 
(\gamma_{1}  + \gamma_{3}) (v_{2} + v_{3}) + 
(\gamma_{4}  + \gamma_{5})(v_{4} + v_{5})\\ &\quad + 
(\gamma_{7}  + \gamma_{8})( v_{7} + v_{8} +  v_{10} + v_{11} +  v_{12} +  v_{13}) \equiv 0.
\end{align*}
From this one gets
\begin{align}\label{ct4} 
\gamma_{1}  + \gamma_{3}= \gamma_{4}  + \gamma_{5}= \gamma_{7}  + \gamma_{8}=  0.
\end{align}
Part ii) of the lemma follows from (\ref{ct1})-(\ref{ct4}). 
\end{proof}

Now, we prove the second part of Theorem \ref{pthm2} for $s = 1$.

Let $f \in (P_5)_9$ such that $[f] \in (QP_5)_{9}^{GL_5}$. Since $[f] \in (QP_5)_{9}^{\Sigma_5}$, using  Lemmas \ref{bdss1} and \ref{bdss2}, we have 
$$f \equiv \gamma_1p(u_1) + \gamma_2p(u_2) + \gamma_3p(u_3) + \gamma_4p_1 + \gamma_5p_2 + \gamma_6p_3,$$
with $\gamma_j \in \mathbb F_2$.
By computing $g_5(f)+f$ in terms of the admissible monomials, we obtain
\begin{align*}g_5(f)+f &\equiv \gamma_1x_2x_4x_5^{7} + \gamma_2x_2^{3}x_3^{3}x_4^{3} + \gamma_3x_2x_4^{3}x_5^{5} + (\gamma_4 + \gamma_5 + \gamma_6) x_2^{3}x_3x_4x_5^{4}\\
&\quad +\gamma_5x_2x_3x_4^{2}x_5^{5} + (\gamma_5 + \gamma_6)x_2x_3x_4x_5^{6} + \text{ other terms} \equiv 0.
\end{align*}
This relation implies $\gamma_j = 0$ for $1 \leqslant j \leqslant 6$. Theorem \ref{pthm2} is proved for $s = 1$.

\medskip
\subsection{The admissible monomials of degree $10$ in $P_5$}\label{subs41}\

\medskip
To prove Theorem \ref{pthm2} for $s=2$, we need to determine all the admissible monomials of degree 10 in $P_5$.

\begin{lems}\label{bdday2}
If $x$  is an admissible monomial of degree $10$ in $P_5$, then $\omega(x)$ is one of the following sequences:
$ (2,2,1), \ (2,4),\ (4,1,1), \ (4,3).$
\end{lems}
\begin{proof}  Observe that $z = x_1^{7}x_2^3$ is the minimal spike of degree $10$ in $P_5$ and  $\omega(z) =  (2,2,1)$. Since $[x] \ne 0$, by Theorem \ref{dlsig}, either $\omega_1(x) =2$ or $\omega_1(x) =4$. If $\omega_1(x) =2$, then $x = x_ix_j y^2$ with $y$ a monomial of degree $4$ in $P_5$ and $i<j$. Since $x$ is admissible, by Theorem \ref{dlcb1}, $y$ is admissible and $y \in P_5^0$.  Using a result in \cite{su1}, one gets either $\omega(y) = (2,1)$ or $\omega(y) = (4,0)$. If $\omega_1(x) =4$, then $x = X_jy_1^2$ with $y_1$ a monomial of degree $3$ in $P_5$. Since $y_1$ is admissible, we see that either $\omega(y_1) =(1,1)$ or $\omega(y_1) =(3,0)$. The lemma is proved.
\end{proof}
 From this lemma and a result in \cite{su1}, we have
\begin{align*}(QP_5)_{10} &= (QP_5^0)_{10} \bigoplus QP_5^+(2,2,1)\\
&\quad \bigoplus QP_5^+(2,4) \bigoplus QP_5^+(4,1,1) \bigoplus QP_5^+(4,3).
\end{align*}

Using a result in \cite{su1}, we have $\dim (QP_5^0)_{10} = 230$. 

\begin{props}\label{mdb10} \

\medskip
{\rm i)} $B_5^+(2,2,1) = \{x_1x_2x_3^{2}x_4^{2}x_5^{4}, x_1x_2x_3^{2}x_4^{4}x_5^{2},  x_1x_2^{2}x_3x_4^{2}x_5^{4},  x_1x_2^{2}x_3x_4^{4}x_5^{2}, x_1x_2^{2}x_3^{4}x_4x_5^{2}\}.$

{\rm ii)} $B_5^+(2,4) = \{x_1x_2^{2}x_3^{2}x_4^{2}x_5^{3},  x_1x_2^{2}x_3^{2}x_4^{3}x_5^{2},  x_1x_2^{2}x_3^{3}x_4^{2}x_5^{2},  x_1x_2^{3}x_3^{2}x_4^{2}x_5^{2},  x_1^{3}x_2x_3^{2}x_4^{2}x_5^{2}\}.$

{\rm iii)} $B_5^+(4,1,1)$ is the set of the following monomials:
 \begin{align*} 
&x_1x_2x_3x_4x_5^{6},\    x_1x_2x_3x_4^{6}x_5,\    x_1x_2x_3^{6}x_4x_5,\    x_1x_2^{6}x_3x_4x_5,\    x_1x_2x_3x_4^{2}x_5^{5},\  \\ & x_1x_2x_3^{2}x_4x_5^{5},\    x_1x_2x_3^{2}x_4^{5}x_5,\    x_1x_2^{2}x_3x_4x_5^{5},\    x_1x_2^{2}x_3x_4^{5}x_5,\    x_1x_2^{2}x_3^{5}x_4x_5,\  \\ & x_1x_2x_3x_4^{3}x_5^{4},\    x_1x_2x_3^{3}x_4x_5^{4},\    x_1x_2x_3^{3}x_4^{4}x_5,\    x_1x_2^{3}x_3x_4x_5^{4},\    x_1x_2^{3}x_3x_4^{4}x_5,\  \\ & x_1x_2^{3}x_3^{4}x_4x_5,\    x_1^{3}x_2x_3x_4x_5^{4},\    x_1^{3}x_2x_3x_4^{4}x_5,\    x_1^{3}x_2x_3^{4}x_4x_5,\    x_1^{3}x_2^{4}x_3x_4x_5.
\end{align*} 

\noindent
{\rm iv)} $B_5^+(4,3)$ is the set of the following monomials:
 \begin{align*} 
&x_1x_2x_3^{2}x_4^{3}x_5^{3},\    x_1x_2x_3^{3}x_4^{2}x_5^{3},\    x_1x_2x_3^{3}x_4^{3}x_5^{2},\    x_1x_2^{2}x_3x_4^{3}x_5^{3},\    x_1x_2^{2}x_3^{3}x_4x_5^{3},\  \\ 
& x_1x_2^{2}x_3^{3}x_4^{3}x_5,\    x_1x_2^{3}x_3x_4^{2}x_5^{3},\    x_1x_2^{3}x_3x_4^{3}x_5^{2},\    x_1x_2^{3}x_3^{2}x_4x_5^{3},\    x_1x_2^{3}x_3^{2}x_4^{3}x_5,\  \\ 
& x_1x_2^{3}x_3^{3}x_4x_5^{2},\    x_1x_2^{3}x_3^{3}x_4^{2}x_5,\    x_1^{3}x_2x_3x_4^{2}x_5^{3},\    x_1^{3}x_2x_3x_4^{3}x_5^{2},\    x_1^{3}x_2x_3^{2}x_4x_5^{3},\  \\ 
& x_1^{3}x_2x_3^{2}x_4^{3}x_5,\    x_1^{3}x_2x_3^{3}x_4x_5^{2},\    x_1^{3}x_2x_3^{3}x_4^{2}x_5,\    x_1^{3}x_2^{3}x_3x_4x_5^{2},\    x_1^{3}x_2^{3}x_3x_4^{2}x_5.
\end{align*} 
\end{props}

From the this proposition and a result in \cite{su1}, we get $\dim(QP_5)_{10} = 280.$ 

The proposition is proved by using Theorems \ref{dlcb1}, \ref{dlsig} and the following.

\begin{lems}\label{inad10} The following monomials are strictly inadmissible:

\medskip
{\rm i)} $x_j^2x_\ell x_t^3,\ j< \ell$; $x_j^2x_\ell x_tx_u^2,\ j < \ell < t$; $x_jx_\ell^2 x_t^2x_u, \ j < \ell < t < u$; $x_1^2x_2x_3x_4x_5$.

{\rm ii)}  $x_j^3x_\ell^4 x_t^3,\ j< \ell<t$; $x_j^2x_\ell^2 x_t^3x_u^3,\ j< \ell<t$; $x_j^2x_\ell x_t^2x_u^2x_v^3,\ j< \ell<t<u$; $x_j^2x_\ell x_tx_u^3x_v^3,\ j< \ell<t$.

Here $(j,\ell,t,u,v)$ is a permutation of $(1,2,3,4,5)$.
\end{lems}
The proof of this lemma is straightforward.

\begin{proof}[Proof of Proposition \ref{mdb10}] We prove the first part of the proposition. The others can be proved by a similar computation. We denote 
 \begin{align*} a_1 &= x_1x_2x_3^{2}x_4^{2}x_5^{4}, a_2 = x_1x_2x_3^{2}x_4^{4}x_5^{2},
 a_3 = x_1x_2^{2}x_3x_4^{2}x_5^{4}, \\ a_4 &= x_1x_2^{2}x_3x_4^{4}x_5^{2}, a_5 = x_1x_2^{2}x_3^{4}x_4x_5^{2}
\end{align*} 

 Let $x$ be an admissible monomial of degree 10 in $P_5$ such that $\omega(x) = (2,2,1)$. Then $x =  x_ix_jy^2$ with $1 \leqslant i < j \leqslant 5$ and $y$ a monomial of degree 4 in $P_5$. Since $x$ is admissible, according to Theorem \ref{dlcb1}, we have $y \in B_5(4)$. 

By a direct computation we that for all $y \in B_5(4)$, such that $ x_ix_jy^2 \ne a_u, \forall u, 1 \leqslant u \leqslant 5$, there is a monomial $w$ which is given in Lemma \ref{inad10}(i) such that $x_ix_jy^2 = wz^{2^{r}}$ with a monomial $z \in P_5$, and $r = \max\{t \in \mathbb Z : \omega_t(w) >0\}$. By Theorem \ref{dlcb1}, $x_ix_jy^2 $ is inadmissible. Since $x = x_ix_jy^2$ is admissible, one gets $x = a_u$ for suitable $u$. 

We now prove the set $\{[a_u],  1 \leqslant u \leqslant 5\}$ is linearly independent in $QP_5$. Suppose that $\mathcal S = \gamma_1a_1 + \gamma_2a_2 + \gamma_3a_3 + \gamma_4a_4 + \gamma_5a_5 \equiv 0$ with $\gamma_u \in \mathbb F_2$. By a simple computation using Theorem \ref{dlsig}, we have
\begin{align*}
p_{(1;2)}(\mathcal S) &\equiv \gamma_3x_1^3x_2x_3^2x_4^4 + \gamma_4x_1^3x_2x_3^4x_2 + \gamma_5x_1^3x_2^4x_3x_4^2 \equiv 0,\\
p_{(2;5)}(\mathcal S) &\equiv (\gamma_1+\gamma_3)x_1x_2x_3^2x_4^6 + \gamma_1x_1x_2^2x_3x_4^6  + \gamma_2x_1x_2^2x_3^4x_4^3 \equiv 0.
\end{align*}
These relations imply $\gamma_u = 0$ for all $u$. The first part of the proposition is proved.
\end{proof}

\subsection{The case $s=2$}\

\medskip
For $s=2$, we have $7.2^s-5 = 23$. Since Kameko's homomorphism 
$$(\widetilde{Sq}^0_*)_{(5,9)}: (QP_5)_{23} \longrightarrow (QP_5)_{9}$$
is an epimorphism, we have 
$(QP_5)_{23} \cong \text{Ker}(\widetilde{Sq}^0_*)_{(5,9)} \bigoplus (QP_5)_{9}$. Hence, we need to compute $\text{Ker}(\widetilde{Sq}^0_*)_{(5,9)}$.

\begin{lems}\label{bdds21}
If $x$  is an admissible monomial of degree $23$ in $P_5$ and $[x] \in \text{\rm Ker}(\widetilde{Sq}^0_*)_{(5,9)}$, then
$\omega(x)$ is one of the following sequences:
$$(3,2,2,1),\ (3,2,4),\ (3,4,1,1),\ (3,4,3).$$
\end{lems}
\begin{proof}  Note that $z = x_1^{15}x_2^7x_3$ is the minimal spike of degree $23$ in $P_5$ and  $\omega(z) =  (3,2,2,1)$. Since $[x] \ne 0$, by Theorem \ref{dlsig}, either $\omega_1(x) =3$ or $\omega_1(x) =5$. If $\omega_1(x) =5$, then $x = X_\emptyset y^2$ with $y$ a monomial of degree $9$ in $P_5$. Since $x$ is admissible, by Theorem \ref{dlcb1}, $y$ is admissible. Hence, $(\widetilde{Sq}^0_*)_{(5,9)}([x]) = [y] \ne 0.$ This contradicts the fact that $[x] \in \text{Ker}(\widetilde{Sq}^0_*)_{(5,9)}$, so $\omega_1(x) =3$. Then, we have $x = x_ix_jx_\ell y_1^2$ with $y_1$ an admissible monomial of degree $10$ in $P_5$. Now, the lemma follows from Lemma \ref{bdday2}.
\end{proof}

Using Lemma \ref{bdds21} and a result in \cite{su1}, we get
\begin{align*}&\text{\rm Ker}(\widetilde{Sq}^0_*)_{(5,9)} =  (QP_5^0)_{23}\bigoplus \big(\text{\rm Ker}(\widetilde{Sq}^0_*)_{(5,9)}\cap (QP_5^+)_{23}\big),\\
&\text{\rm Ker}(\widetilde{Sq}^0_*)_{(5,9)}\cap (QP_5^+)_{23} =  \bigoplus_{j=1}^4 QP_5^+(\omega_{(j)}).
\end{align*}
 Here $\omega_{(1)} = (3,2,2,1),\ \omega_{(2)} = (3,4,1,1),\ \omega_{(3)} = (3,4,3),\ \omega_{(4)} = (3,2,4)$.
From a result in \cite{su1}, we easily obtain $\dim(QP_5^0)_{23} = 635$. In this subsection, we prove the following.

\begin{props}\label{mds20} The set $\{[b_t]  : 1 \leqslant t \leqslant 419\}$ is a basis of the $\mathbb F_2$-vector space  $\text{\rm Ker}(\widetilde{Sq}^0_*)_{(5,9)}\cap (QP_5^+)_{23}$. Here the monomials $b_t = b_{23,t}$, with $1 \leqslant t \leqslant 419$, are determined as in Subsection \ref{s73}. 
\end{props}
 
We prove this proposition by proving some lemmas.

\begin{lems}\label{mds21} The space $QP_5^+(\omega_{(1)})$ is spanned by the set $\{[a_t]  : 1 \leqslant t \leqslant 290\}$. 
\end{lems}

The following lemma is proved by a direct computation.
\begin{lems}\label{bdd21} The following monomials are strictly inadmissible:

\medskip
{\rm i)} $x_j^2x_\ell x_tx_u^3,\  j < \ell < t;\ x_j^2x_\ell x_tx_ux_v^2,\ j < \ell< t < u;\ x_1x_2^2x_3^2x_4x_5.$
Here $(j,\ell, t,u,v)$ is a permutation of $(1,2,3,4,5)$.

{\rm ii)} $f_i(\bar x)$, $1 \leqslant i \leqslant 5$, where $\bar x$ is one of the following monomials:
\begin{align*} &x_1^{3}x_2^{12}x_3x_4^{7},\  x_1^{3}x_2^{12}x_3^{7}x_4,\  x_1^{3}x_2^{12}x_3^{3}x_4^{5},\  x_1^{3}x_2^{4}x_3^{9}x_4^{7},\\ 
&  x_1^{3}x_2^{5}x_3^{9}x_4^{6}, \ x_1^{3}x_2^{5}x_3^{8}x_4^{7},\  x_1^{7}x_2^{8}x_3^{3}x_4^{5}.
\end{align*}
\end{lems}

\begin{lems}\label{bdd21s} The following monomials are strictly inadmissible:
\begin{align*} &x_1x_2^{6}x_3^{8}x_4^{3}x_5^{5}, \ x_1^{3}x_2^{4}x_3x_4^{8}x_5^{7}, \ x_1^{3}x_2^{4}x_3x_4^{9}x_5^{6}, \ x_1^{3}x_2^{4}x_3^{3}x_4^{4}x_5^{9}, \ x_1^{3}x_2^{4}x_3^{3}x_4^{12}x_5,\\ & x_1^{3}x_2^{4}x_3^{8}x_4x_5^{7}, \ x_1^{3}x_2^{4}x_3^{8}x_4^{3}x_5^{5}, \ x_1^{3}x_2^{4}x_3^{8}x_4^{7}x_5, \ x_1^{3}x_2^{4}x_3^{9}x_4x_5^{6}, \ x_1^{3}x_2^{4}x_3^{9}x_4^{6}x_5,\\ & x_1^{3}x_2^{4}x_3^{11}x_4^{4}x_5, \ x_1^{3}x_2^{5}x_3x_4^{8}x_5^{6}, \ x_1^{3}x_2^{5}x_3^{8}x_4x_5^{6}, \ x_1^{3}x_2^{5}x_3^{8}x_4^{6}x_5, \ x_1^{3}x_2^{12}x_3x_4x_5^{6},\\ & x_1^{3}x_2^{12}x_3x_4^{6}x_5, \ x_1^{3}x_2^{12}x_3^{3}x_4^{4}x_5, \ x_1^{7}x_2^{8}x_3^{3}x_4^{4}x_5.
\end{align*}
\end{lems}
\begin{proof} We prove the lemma for the monomial $x = x_1x_2^{6}x_3^{8}x_4^{3}x_5^{5}$. The others can be proved by a similar computation. By a direct computation, we have
\begin{align*}
x &= x_1x_2^{4}x_3^{10}x_4^{3}x_5^{5} + x_1x_2^{4}x_3^{6}x_4^{3}x_5^{9} + x_1x_2^{6}x_3^{3}x_4^{8}x_5^{5} + x_1x_2^{8}x_3^{3}x_4^{6}x_5^{5} + x_1x_2^{6}x_3^{3}x_4^{5}x_5^{8}\\
&\quad + x_1x_2^{8}x_3^{3}x_4^{5}x_5^{6} + x_1x_2^{3}x_3^{6}x_4^{8}x_5^{5} + x_1x_2^{3}x_3^{8}x_4^{6}x_5^{5} + x_1x_2^{3}x_3^{8}x_4^{5}x_5^{6} + x_1x_2^{3}x_3^{6}x_4^{5}x_5^{8}\\ 
&\quad+ x_1x_2^{3}x_3^{5}x_4^{8}x_5^{6} + x_1x_2^{3}x_3^{5}x_4^{6}x_5^{8} + Sq^1(x_1^2x_2^{5}x_3^{5}x_4^{5}x_5^{5}) + Sq^2(x_1x_2^{6}x_3^{6}x_4^{3}x_5^{5}\\
&\quad + x_1x_2^{5}x_3^{5}x_4^{5}x_5^{5} + x_1x_2^{6}x_3^{3}x_4^{6}x_5^{5} + x_1x_2^{6}x_3^{3}x_4^{5}x_5^{6} + x_1x_2^{3}x_3^{6}x_4^{6}x_5^{5} + x_1x_2^{3}x_3^{6}x_4^{5}x_5^{6}\\
&\quad + x_1x_2^{3}x_3^{5}x_4^{6}x_5^{6}) + Sq^4(x_1x_2^{4}x_3^{6}x_4^{3}x_5^{5}) \quad \text{mod}(P_5^-(3,2,2,1)). 
\end{align*}
Hence, the monomial $x$ is strictly inadmissible.
\end{proof}

\begin{proof}[Proof of Lemma \ref{mds21}] Let $x$ be an admissible monomial in the space $P_5^+$ such that $\omega(x) = \omega_{(1)}$. Then $x = x_jx_\ell x_ty^2$ with $y \in B_5(2,2,1)$. 

Let $z \in B_5(2,2,1)$ such that $x_jx_\ell x_tz^2 \in P_5^+$.
By a direct computation using the results in Subsection \ref{subs41}, we see that if $x_jx_\ell x_tz^2\ne b_t, \forall t, 1 \leqslant t \leqslant 290$, then there is a monomial $w$ which is given in Lemma \ref{bdd21} such that $x_jx_\ell x_tz^2= wz_1^{2^{u}}$ with suitable monomial $z_1 \in P_5$, and $u = \max\{j \in \mathbb Z : \omega_j(w) >0\}$. By Theorem \ref{dlcb1}, $x_jx_\ell x_tz^2$ is inadmissible. Since $x = x_jx_\ell x_ty^2$ with $y \in B_5(2,2,1)$ and $x$ is admissible, one gets $x= b_t$ for suitable $t$. This implies $B_5^+(\omega_{(1)}) \subset \{b_t  : 1 \leqslant t \leqslant 290\}$. 
The proposition follows. 
\end{proof}
\begin{lems}\label{mds22} $B_5(\omega_{(4)}) = B_5^+(\omega_{(4)}) = \emptyset$. That means $QP_5(\omega_{(4)}) = 0$. 
\end{lems}
\begin{proof} Let $x$ be an admissible monomial in $P_5^+$ such that $\omega(x) = \omega_{(4)}$. Then $x = x_jx_\ell x_ty^2$ with $y \in B_5(2,4)$. By a direct computation using Theorem \ref{dlcb1}, Proposition \ref{mdb10} and Lemma \ref{bdd21}, we see that $x$ is a permutation of one of the monomials: $x_1^3x_2^{4}x_3^{4}x_4^{5}x_5^{7}$, $x_1^3x_2^{4}x_3^{5}x_4^{5}x_5^{6}$. A simple computation shows:
\begin{align*} x_1^3x_2^{4}x_3^{4}x_4^{5}x_5^{7} &= Sq^1(x_1^3x_2x_3^{2}x_4^{9}x_5^{7} + x_1^3x_2x_3^{2}x_4^{5}x_5^{11}) + Sq^2(x_1^5x_2^{2}x_3^{2}x_4^{5}x_5^{7} + x_1^5x_2x_3^{2}x_4^{6}x_5^{7})\\ 
&\quad+ Sq^4(x_1^3x_2^{2}x_3^{2}x_4^{5}x_5^{7} + x_1^3x_2x_3^{2}x_4^{6}x_5^{7}) \ \text{ mod}(P_5^-(B_5(\omega_{(4)}))).
\end{align*}
This relation implies $[ x_1^3x_2^{4}x_3^{4}x_4^{5}x_5^{7}]_{\omega_{(4)}} =0$. By a similar computation, we have $[ x_1^3x_2^{4}x_3^{5}x_4^{5}x_5^{6}]_{\omega_{(4)}} =0$. The proposition is proved.
\end{proof}
\begin{lems}\label{mds23} The space $QP_5(\omega_{(2)})$ is spanned by the set $\{[b_t]_{\omega_{(2)}}  : 291 \leqslant t \leqslant 395\}$, where the monomials $b_t$ are determined as in Subsection \ref{s73}.
\end{lems}

\begin{lems}\label{bdd22} The following monomials are strictly inadmissible:

\medskip
{\rm i)} $x_j^2x_\ell x_t^2 x_u^3x_v^3,  i < j$; \ $x_j^2x_\ell^3x_t^3 x_u^3.$ 

Here $(j,\ell,t,u,v) $ is a permutation of $(1,2,3,4,5)$.

{\rm ii)} $x_1x_2^{2}x_3^{6}x_4^{3}x_5^{3},\  x_1x_2^{6}x_3^{2}x_4^{3}x_5^{3},\  x_1x_2^{6}x_3^{3}x_4^{2}x_5^{3},\  x_1x_2^{6}x_3^{3}x_4^{3}x_5^{2}.$
\end{lems}

The proof of this lemma is straightforward.

\begin{proof}[Proof of Lemma \ref{mds23}] Let $x$ be an admissible monomial in $P_5^+$ such that $\omega(x) = \omega_{(2)}$. Then $x = x_jx_\ell x_ty^2$ with $y \in B_5(4,1,1)$. 

Let $z \in B_5(4,1,1)$ such that $ x_jx_\ell x_tz^2 \in P_5^+$. 
By a direct computation using Proposition \ref{mdb10}, we see that if $x_jx_\ell x_tz^2\notin  b_t, \ \forall t, 291 \leqslant t \leqslant 395$, then there is a monomial $w$ which is given in Lemma \ref{bdd22} such that $x_jx_\ell x_tz^2= wz_1^{2^{u}}$ with suitable monomial $z_1 \in P_5$, and $u = \max\{j \in \mathbb Z : \omega_j(w) >0\}$. By Theorem \ref{dlcb1}, $x_jx_\ell x_tz^2$ is inadmissible. Since $x = x_jx_\ell x_ty^2$ with $y \in B_5(4,1,1)$ and $x$ is admissible, one gets $x= b_t$ for some $t$. The proposition is proved 
\end{proof}

\begin{lems}\label{mds24} The space $QP_5(\omega_{(3)})$ is spanned by the set $\{[a_t]_{\omega_{(3)}}  : 396 \leqslant t \leqslant 419\}$, where the monomials $b_t$ are determined as in Subsection \ref{s73}.
\end{lems}
 The following lemma is proved by a direct computation.
\begin{lems}\label{bdd23} The following monomials are strictly inadmissible:
 $$ x_jx_\ell^6x_t^3 x_u^6x_v^7, \ j < \ell < t;\ \ x_jx_\ell^2x_t^6 x_u^7x_v^7.$$
Here $(j,\ell,t,u,v) $ is a permutation of $(1,2,3,4,5)$.
\end{lems}

\begin{proof}[Proof of Lemma \ref{mds24}] Let $x$ be an admissible monomial in $P_5^+$ such that $\omega(x) = \omega_{(3)}$. Then $x = x_jx_\ell x_ty^2$ with $y \in B_5(4,3)$. 

Let $z \in B_5(4,3)$ such that $ x_jx_\ell x_tz^2 \in P_5^+$. 
By a direct computation using Proposition \ref{mdb10}, we see that if $x_jx_\ell x_tz^2\ne  b_t, \forall t, 396 \leqslant t \leqslant 419$, then there is a monomial $w$ which is given in Lemma \ref{bdd23} such that $x_jx_\ell x_tz^2= wz_1^{2^{u}}$ with suitable monomial $z_1 \in P_5$, and $u = \max\{j \in \mathbb Z : \omega_j(w) >0\}$. By Theorem \ref{dlcb1}, $x_jx_\ell x_tz^2$ is inadmissible. Since $x = x_jx_\ell x_ty^2$ with $y \in B_5(4,3)$ and $x$ is admissible, one gets $x = b_t$ for some $t,\ 396 \leqslant t \leqslant 419$. This implies $B_5^+(\omega_{(3)}) \subset \{b_t : 396 \leqslant t \leqslant 419\}$. The proposition follows. 
\end{proof}

\begin{proof}[Proof of Proposition \ref{mds20}] From Lemmas \ref{mds21}, \ref{mds22}, \ref{mds23} and \ref{mds24} we see that  the $\mathbb F_2$-vector space  $\text{\rm Ker}(\widetilde{Sq}^0_*)_{(5,9)}\cap (QP_5^+)_{23}$ is spanned by the set $\{[b_t]  : 1 \leqslant t \leqslant 419\}$.  Now we prove that the set $\{[b_t]  : 1 \leqslant t \leqslant 419\}$ is linearly independent in $QP_5$. Suppose there is a linear relation
$$\mathcal S = \sum_{t = 1}^{419}\gamma_tb_t \equiv 0,$$ 
where $\gamma_t \in \mathbb F_2$. We explicitly compute $p_{(i;j)}(\mathcal S)$ in terms of the admissible monomials in $P_4$. From the relations $p_{(i;j)}(\mathcal S) \equiv 0$ with $1 \leqslant i < j \leqslant 5$, one gets $\gamma_t=0$ for all $1 \leqslant t \leqslant 419.$
\end{proof}

\begin{corls} Under the above notations, we have
$$\dim QP_5(\omega_{(1)}) =  925,\ \dim QP_5(\omega_{(2)}) =  105,\ \dim QP_5(\omega_{(3)}) =  24.$$
\end{corls}

Now we compute $(QP_5)_{23}^{GL_5}$. Since $(QP_5)_{9}^{GL_5} = 0$, using Theorem \ref{dlcb1}, we have
$(QP_5)_{23}^{GL_5} = \text{\rm Ker}(\widetilde{Sq}^0_*)_{(5,9)}^{GL_5}.$
Recall that 
\begin{align*}\text{\rm Ker}(\widetilde{Sq}^0_*)_{(5,9)} &= QP_5(\omega_{(1)}) \bigoplus QP_5(\omega_{(2)})\bigoplus QP_5(\omega_{(3)}).
\end{align*}

By Lemma \ref{mds24}, $\dim QP_5(\omega_{(3)}) = 24$ with the basis $[B_5(\bar b_1)]_{\omega_{(3)}}\cup [B_5(\bar b_2)]_{\omega_{(3)}}$, where $\bar b_{1} =  x_1x_2^{3}x_3^{6}x_4^{6}x_5^{7},\ \bar b_{2} =  x_1^{3}x_2^{3}x_3^{5}x_4^{6}x_5^{6}.$
\begin{props}\label{mdbs21}
$QP_5(\omega_{(3)})^{GL_5} = 0$.
\end{props}

By a direct computation we easily obtain the following lemma.
\begin{lems}\label{bdss21} We have a direct summand decomposition of the $\Sigma_5$-modules:
$$QP_5(\omega_{(3)}) = \langle [\Sigma_5(\bar b_1)]_{\omega_{(3)}}\rangle \bigoplus \langle [\Sigma_5(\bar b_2)]_{\omega_{(3)}}\rangle.$$
\end{lems}

\begin{lems}\label{bdss22}
$\langle [\Sigma_5(\bar b_1)]_{\omega_{(3)}}\rangle^{\Sigma_5} = \langle [p(\bar b_1)]_{\omega_{(3)}}\rangle$ and $\langle [\Sigma_5(\bar b_2)]_{\omega_{(3)}}\rangle^{\Sigma_5} = 0$.
\end{lems}
\begin{proof} From Lemma \ref{mds24}, we see that $\dim\langle[\Sigma_5(\bar b_2)]_{\omega_{(3)}}\rangle =4$ with a basis consisting of the classes represented by the following monomials:
\begin{align*}u_{1} =  x_1^{3}x_2^{3}x_3^{5}x_4^{6}x_5^{6},\  u_{2} =  x_1^{3}x_2^{5}x_3^{3}x_4^{6}x_5^{6},\  u_{3} =  x_1^{3}x_2^{5}x_3^{6}x_4^{3}x_5^{6},\  u_{4} =  x_1^{3}x_2^{5}x_3^{6}x_4^{6}x_5^{3}.
\end{align*}
Suppose $f = \sum_{t=1}^{4}\gamma_tu_t$ with $\gamma_t \in \mathbb F_2$ and $[f] \in \langle[\Sigma_5(\bar b_2)]_{\omega_{(3)}}\rangle^{\Sigma_5}$. By a direct computation, we have
\begin{align*} g_1(f) + f &\equiv_{\omega_{(3)}} (\gamma_{2}  + \gamma_{3} + \gamma_{4}) u_{1}   \equiv_{\omega_{(3)}}  0,\\
 g_2(f) + f &\equiv_{\omega_{(3)}} (\gamma_{1}  + \gamma_{2}) u_{1}  \equiv_{\omega_{(3)}}  0,\\ 
g_3(f) + f &\equiv_{\omega_{(3)}} (\gamma_{2}  + \gamma_{3}) u_{2} \equiv_{\omega_{(3)}}  0,\\
 g_4(f) + f &\equiv_{\omega_{(3)}} (\gamma_{3}  + \gamma_{4}) u_{3}  \equiv_{\omega_{(3)}}  0.
\end{align*}
From the above relations one gets $\gamma_t =0$ for $t=1,2,3,4$. By a similar computation we obtain $\langle [\Sigma_5(\bar b_1)]_{\omega_{(3)}}\rangle^{\Sigma_5} = \langle [p(\bar b_1)]_{\omega_{(3)}}\rangle$.
\end{proof}
\begin{proof}[Proof of Proposition \ref{mdbs21}] Let $f \in P_5(\omega_{(3)})$ such that $[f]_{\omega_{(3)}} \in QP_5(\omega_{(3)})^{GL_5}$. Since $[f]_{\omega_{(3)}} \in QP_5(\omega_{(3)})^{\Sigma_5}$, using Lemmas \ref{bdss21},  and \ref{bdss22}, we have $f \equiv_{\omega_{(3)}} \gamma p(\bar b_1)$ with $\gamma \in \mathbb F_2$. By computing $g_5(f)+f$ in terms of the admissible monomials, we obtain
\begin{align*}g_5(f)+f \equiv_{\omega_{(3)}} \gamma \bar b_1 + \text{ other terms} \equiv_{\omega_{(3)}} 0.
\end{align*}
This relation implies $\gamma = 0$. The proposition is proved.
\end{proof}

\begin{props}\label{mdbs22}
$QP_5(\omega_{(2)})^{GL_5} = 0$.
\end{props}

By computing from Lemma \ref{mds23}, we see that $\dim QP_5(\omega_{(2)}) = 105$ with the basis $\bigcup_{j=1}^4[B(a_j)]_{\omega_{(2)}}$, where
\begin{align*}&a_{1} =  x_1x_2^{2}x_3^{2}x_4^{3}x_5^{15},\ a_{2} =  x_1x_2^{2}x_3^{2}x_4^{7}x_5^{11},\ a_{3} =  x_1x_2^{3}x_3^{3}x_4^{6}x_5^{10},\ a_{4} =  x_1x_2^{2}x_3^{3}x_4^{6}x_5^{11}.
\end{align*}
 By a direct computation, we obtain the following.

\begin{lems}\label{bdss23} We have a direct summand decomposition of the $\Sigma_5$-modules: 
$$QP_5(\omega_{(2)}) = \bigoplus_{j=1}^4 \langle [\Sigma_5(a_j)]_{\omega_{(2)}}\rangle.$$
\end{lems}

\begin{lems}\label{bdss24}
$\langle [\Sigma_5(a_j)]_{\omega_{(2)}}\rangle^{\Sigma_5} = \langle [p(a_j)]_{\omega_{(2)}}\rangle,\ j =1,2,3,$ and $\langle [\Sigma_5(a_4)]_{\omega_{(2)}}\rangle^{\Sigma_5} = \langle [p_4]_{\omega_{(2)}}]\rangle$, where the polynomial $p_4$ is explicitly determined as in the Subsection \ref{s74}.
\end{lems}

\begin{proof}[Proof of Proposition \ref{mdbs22}] Let $f \in P_5(\omega_{(2)})$ such that $[f]_{\omega_{(2)}} \in QP_5(\omega_{(2)})^{GL_5}$. Since $[f]_{\omega_{(2)}} \in QP_5(\omega_{(2)})^{\Sigma_5}$, using Lemmas \ref{bdss23},  and \ref{bdss24}, we have 
$$f \equiv_{\omega_{(2)}} \gamma_1p(a_1) +  \gamma_2p(a_2) +  \gamma_3p(a_3) +  \gamma_4p_4,$$ 
with $\gamma_j \in \mathbb F_2$. By computing $g_5(f)+f$ in terms of the admissible monomials, we obtain
\begin{align*}g_5(f)+f &\equiv_{\omega_{(2)}} \gamma_1 x_1^{7}x_2^{3}x_3x_4^{2}x_5^{10} + \gamma_2 x_1^{3}x_2^{3}x_3x_4^{2}x_5^{14} + \gamma_3 x_1x_2^{7}x_3^{2}x_4^{3}x_5^{10}\\ 
&\hskip2cm+ (\gamma_1 + \gamma_4) x_1^{3}x_2^{13}x_3^{2}x_4^{2}x_5^{3} +  \text{ other terms} \equiv_{\omega_{(2)}} 0.
\end{align*}
This relation implies $\gamma_j = 0$ with $j = 1,2,3,4$. The proposition  is proved.
\end{proof}

\begin{props}\label{mdbs23}
$QP_5(\omega_{(1)})^{GL_5} = 0$.
\end{props}

By using Proposition \ref{mds22}, we see that $\dim QP_5(\omega_{(1)}) = 925$. Consider the following monomials:
\begin{align*} &c_{1} =  x_1x_2^{7}x_3^{15}, \ c_{2} =  x_1^{3}x_2^{5}x_3^{15}, \ c_{3} =  x_1^{3}x_2^{7}x_3^{13}, \ c_{4} =  x_1x_2^{2}x_3^{5}x_4^{15}, \ c_{5} =  x_1x_2x_3^{2}x_4^{4}x_5^{15}, \\
& c_{6} =  x_1x_2^{2}x_3^{7}x_4^{13}, \ c_{7} =  x_1x_2^{3}x_3^{5}x_4^{14}, \ c_{8} =  x_1x_2^{3}x_3^{6}x_4^{13}, \ c_{9} =  x_1x_2^{3}x_3^{7}x_4^{12}, \\
& c_{10} =  x_1^{3}x_2^{5}x_3^{6}x_4^{9}, \ c_{11} =  x_1x_2x_3^{2}x_4^{6}x_5^{13}, \ c_{12} =  x_1x_2x_3^{2}x_4^{7}x_5^{12}, \ c_{13} =  x_1x_2x_3^{3}x_4^{6}x_5^{12}, \\ 
&c_{14} =  x_1x_2^{2}x_3^{3}x_4^{4}x_5^{13}, \ c_{15} =  x_1x_2^{2}x_3^{3}x_4^{5}x_5^{12}, \ c_{16} =  x_1x_2^{2}x_3^{5}x_4^{6}x_5^{9}, \ c_{17} =  x_1x_2^{2}x_3^{5}x_4^{7}x_5^{8}.
\end{align*}

\begin{lems}\label{bdss25} We have a direct summand decomposition of the $\Sigma_5$-modules: 
$$QP_5(\omega_{(1)}) = \bigoplus_{j=1}^5 \langle [\Sigma_5(a_j)]\rangle\bigoplus\langle [\Sigma_5(c_6, \ldots, c_{10})]\rangle\bigoplus \langle [\Sigma_5(c_{11}, \ldots, c_{17})]\rangle .$$
\end{lems}

\begin{lems}\label{bdss26}\

{\rm i)} $\langle [\Sigma_5(a_j)]\rangle^{\Sigma_5} = \langle [p(a_j)]\rangle,\ j =1,2,3,4,5.$

{\rm ii)} $\langle [\Sigma_5(c_{11}, \ldots, c_{17})]\rangle^{\Sigma_5} = 0$.

{\rm iii)} $\langle [\Sigma_5(c_{6}, \ldots, c_{10})]\rangle^{\Sigma_5} = \langle [p_5+p_6], [p_6+p_7]\rangle$, where the polynomials $p_5, p_6, p_7$ are determined as in Subsection \ref{s74}.
\end{lems}
The proofs of the above lemmas are straightforward.

\begin{proof}[Proof of Proposition \ref{mdbs23}] Let $f \in P_5(\omega_{(1)})$ such that $[f] \in QP_5(\omega_{(1)})^{GL_5}$. Since $[f] \in QP_5(\omega_{(2)})^{\Sigma_5}$, using Lemmas \ref{bdss25},  and \ref{bdss26}, we have 
$$f \equiv \sum_{j=1}^5\gamma_jp(c_j) +  \gamma_6(p_5+p_6) +  \gamma_7(p_6+p_7),$$ 
with $\gamma_j \in \mathbb F_2$, $1\leqslant j \leqslant 7$. By computing $g_5(f)+f$ in terms of the admissible monomials and using Theorem \ref{dlsig}, we obtain
\begin{align*}g_5(f)+f &\equiv \gamma_{1}x_2x_3^{7}x_4^{15} + \gamma_{2}x_2^{3}x_3^{15}x_4^{5} + \gamma_{3}x_1^{3}x_2^{7}x_3^{13} + \gamma_{4}x_2x_3x_4^{6}x_5^{15}\\
&\quad + \gamma_{5}x_1x_2^{15}x_3x_4^{2}x_5^{4} + \gamma_{6}x_2x_3x_4^{7}x_5^{14} + \gamma_{7}x_2x_3^{3}x_4^{5}x_5^{14}   + \text{ other terms} \equiv 0.
\end{align*}
This relation implies $\gamma_j = 0$ with $j = 1,2,\ldots , 7$.  The proposition follows.
\end{proof}

\section{Appendix} \label{s7}

In the appendix, we list all admissible monomials of degrees 9, 10, 23 in $P_4$ and $P_5$. We order a set $B$ of some monomials of degree $n$ in $P_k$ by using the order as in Definition \ref{defn3}.

\medskip\noindent
\subsection{The admissible monomials of degree 9 in $P_5$.}\ \label{s71}

\subsubsection{The admissible monomials of degree 9 in $P_4$.}\

\medskip
$B_4(9)$ is the set of 46 monomials $a_t = a_{9,t}, \ 1 \leqslant t \leqslant 46$:

\medskip
 \centerline{\begin{tabular}{lllll}
$1.\ \   x_2x_3x_4^{7}$& $2.\ \   x_2x_3^{3}x_4^{5}$& $3.\ \   x_2x_3^{7}x_4$& $4.\ \   x_2^{3}x_3x_4^{5}$& $5.\ \   x_2^{3}x_3^{5}x_4$\cr $6.\ \   x_2^{7}x_3x_4$& $7.\ \   x_1x_3x_4^{7}$& $8.\ \   x_1x_3^{3}x_4^{5}$& $9.\ \   x_1x_3^{7}x_4$& $10.\ \   x_1x_2x_4^{7}$\cr $11.\ \   x_1x_2x_3x_4^{6}$& $12.\ \   x_1x_2x_3^{2}x_4^{5}$& $13.\ \   x_1x_2x_3^{3}x_4^{4}$& $14.\ \   x_1x_2x_3^{6}x_4$& $15.\ \   x_1x_2x_3^{7}$\cr $16.\ \   x_1x_2^{2}x_3x_4^{5}$& $17.\ \   x_1x_2^{2}x_3^{5}x_4$& $18.\ \   x_1x_2^{3}x_4^{5}$& $19.\ \   x_1x_2^{3}x_3x_4^{4}$& $20.\ \   x_1x_2^{3}x_3^{4}x_4$\cr $21.\ \   x_1x_2^{3}x_3^{5}$& $22.\ \   x_1x_2^{6}x_3x_4$& $23.\ \   x_1x_2^{7}x_4$& $24.\ \   x_1x_2^{7}x_3$& $25.\ \   x_1^{3}x_3x_4^{5}$\cr $26.\ \   x_1^{3}x_3^{5}x_4$& $27.\ \   x_1^{3}x_2x_4^{5}$& $28.\ \   x_1^{3}x_2x_3x_4^{4}$& $29.\ \   x_1^{3}x_2x_3^{4}x_4$& $30.\ \   x_1^{3}x_2x_3^{5}$\cr $31.\ \   x_1^{3}x_2^{4}x_3x_4$& $32.\ \   x_1^{3}x_2^{5}x_4$& $33.\ \   x_1^{3}x_2^{5}x_3$& $34.\ \   x_1^{7}x_3x_4$& $35.\ \   x_1^{7}x_2x_4$\cr $36.\ \   x_1^{7}x_2x_3$& $37.\ \   x_2^{3}x_3^{3}x_4^{3}$& $38.\ \   x_1x_2^{2}x_3^{3}x_4^{3}$& $39.\ \   x_1x_2^{3}x_3^{2}x_4^{3}$& $40.\ \   x_1x_2^{3}x_3^{3}x_4^{2}$\cr $41.\ \   x_1^{3}x_3^{3}x_4^{3}$& $42.\ \   x_1^{3}x_2x_3^{2}x_4^{3}$& $43.\ \   x_1^{3}x_2x_3^{3}x_4^{2}$& $44.\ \   x_1^{3}x_2^{3}x_4^{3}$& $45.\ \   x_1^{3}x_2^{3}x_3x_4^{2}$\cr $46.\ \   x_1^{3}x_2^{3}x_3^{3}$.& 
\end{tabular}}

\medskip
\subsubsection{The admissible monomials of degree 9 in ${P_5}$.}\

\medskip 
$B_5(9) = B_5^0(9) \cup B_5^+(9)$, where $B_5^0(9) = \Phi^0(B_4(9))$, $|B_5^0(9)| =160$ and $B_5^+(9)= B_5^+(3,1,1)\cup B_5^+(3,3)\cup B_5(5,2)$, where.

\medskip
{$B_5^+(3,1,1)$ is the set of 6 monomials $b_t = b_{9,t}, 1 \leqslant t \leqslant 6$}:

\medskip
 \centerline{\begin{tabular}{llll}
$1.\ \   x_1x_2x_3x_4^{2}x_5^{4}$& $2.\ \   x_1x_2x_3^{2}x_4x_5^{4}$& $3.\ \   x_1x_2x_3^{2}x_4^{4}x_5$& $4.\ \   x_1x_2^{2}x_3x_4x_5^{4}$\cr 
$5.\ \   x_1x_2^{2}x_3x_4^{4}x_5$&$6.\ \   x_1x_2^{2}x_3^{4}x_4x_5$.&
\end{tabular}}

\medskip
{$B_5^+(3,3)$ is the set of 15 monomials $b_t = b_{9,t}, 7 \leqslant t \leqslant 21$}:

\medskip
 \centerline{\begin{tabular}{llll}
  $7.\ \   x_1x_2x_3^{2}x_4^{2}x_5^{3}$& $8.\ \   x_1x_2x_3^{2}x_4^{3}x_5^{2}$& $9.\ \   x_1x_2x_3^{3}x_4^{2}x_5^{2}$& $10.\ \   x_1x_2^{2}x_3x_4^{2}x_5^{3}$\cr 
$11.\ \   x_1x_2^{2}x_3x_4^{3}x_5^{2}$& $12.\ \   x_1x_2^{2}x_3^{2}x_4x_5^{3}$& $13.\ \   x_1x_2^{2}x_3^{2}x_4^{3}x_5$& $14.\ \   x_1x_2^{2}x_3^{3}x_4x_5^{2}$\cr 
$15.\ \   x_1x_2^{2}x_3^{3}x_4^{2}x_5$& $16.\ \   x_1x_2^{3}x_3x_4^{2}x_5^{2}$& $17.\ \   x_1x_2^{3}x_3^{2}x_4x_5^{2}$& $18.\ \   x_1x_2^{3}x_3^{2}x_4^{2}x_5$\cr 
$19.\ \   x_1^{3}x_2x_3x_4^{2}x_5^{2}$& $20.\ \   x_1^{3}x_2x_3^{2}x_4x_5^{2}$& $21.\ \   x_1^{3}x_2x_3^{2}x_4^{2}x_5$.& 
\end{tabular}}

\medskip
{$B_5(5,2)$ is the set of 10 monomials $b_t = b_{9,t}, 22 \leqslant t \leqslant 31$}:

\medskip
 \centerline{\begin{tabular}{llll}
$22.\ \   x_1x_2x_3x_4^{3}x_5^{3}$& $23.\ \   x_1x_2x_3^{3}x_4x_5^{3}$& $24.\ \   x_1x_2x_3^{3}x_4^{3}x_5$& $25.\ \   x_1x_2^{3}x_3x_4x_5^{3}$\cr 
$26.\ \   x_1x_2^{3}x_3x_4^{3}x_5$& $27.\ \   x_1x_2^{3}x_3^{3}x_4x_5$& $28.\ \   x_1^{3}x_2x_3x_4x_5^{3}$& $29.\ \   x_1^{3}x_2x_3x_4^{3}x_5$\cr 
$30.\ \   x_1^{3}x_2x_3^{3}x_4x_5$& $31.\ \   x_1^{3}x_2^{3}x_3x_4x_5$.& 
\end{tabular}}

\medskip
\subsection{The admissible monomials of degree 10 in $ {P_5}$.}\ \label{s72}

\medskip 
\subsubsection{The admissible monomials of degree 10 in $ {P_4}$.}\ 

\medskip
$B_4(10)$ is the set of 70 monomials.

\medskip
 \centerline{\begin{tabular}{lllll}
1.\  $x_3^{3}x_4^{7} $& 2.\  $x_3^{7}x_4^{3} $& 3.\  $x_2x_3^{2}x_4^{7} $& 4.\  $x_2x_3^{3}x_4^{6} $& 5.\  $x_2x_3^{6}x_4^{3}$\cr 
6.\  $x_2x_3^{7}x_4^{2} $& 7.\  $x_2^{3}x_4^{7} $& 8.\  $x_2^{3}x_3x_4^{6} $& 9.\  $x_2^{3}x_3^{3}x_4^{4} $& 10.\  $x_2^{3}x_3^{5}x_4^{2}$\cr 
11.\  $x_2^{3}x_3^{7} $& 12.\  $x_2^{7}x_4^{3} $& 13.\  $x_2^{7}x_3x_4^{2} $& 14.\  $x_2^{7}x_3^{3} $& 15.\  $x_1x_3^{2}x_4^{7}$\cr 
16.\  $x_1x_3^{3}x_4^{6} $& 17.\  $x_1x_3^{6}x_4^{3} $& 18.\  $x_1x_3^{7}x_4^{2} $& 19.\  $x_1x_2x_3^{2}x_4^{6} $& 20.\  $x_1x_2x_3^{6}x_4^{2}$\cr 
21.\  $x_1x_2^{2}x_4^{7} $& 22.\  $x_1x_2^{2}x_3x_4^{6} $& 23.\  $x_1x_2^{2}x_3^{3}x_4^{4} $& 24.\  $x_1x_2^{2}x_3^{4}x_4^{3} $& 25.\  $x_1x_2^{2}x_3^{5}x_4^{2}$\cr 
26.\  $x_1x_2^{2}x_3^{7} $& 27.\  $x_1x_2^{3}x_4^{6} $& 28.\  $x_1x_2^{3}x_3^{2}x_4^{4} $& 29.\  $x_1x_2^{3}x_3^{4}x_4^{2} $& 30.\  $x_1x_2^{3}x_3^{6}$\cr 
31.\  $x_1x_2^{6}x_4^{3} $& 32.\  $x_1x_2^{6}x_3x_4^{2} $& 33.\  $x_1x_2^{6}x_3^{3} $& 34.\  $x_1x_2^{7}x_4^{2} $& 35.\  $x_1x_2^{7}x_3^{2}$\cr 
36.\  $x_1^{3}x_4^{7} $& 37.\  $x_1^{3}x_3x_4^{6} $& 38.\  $x_1^{3}x_3^{3}x_4^{4} $& 39.\  $x_1^{3}x_3^{5}x_4^{2} $& 40.\  $x_1^{3}x_3^{7}$\cr 
41.\  $x_1^{3}x_2x_4^{6} $& 42.\  $x_1^{3}x_2x_3^{2}x_4^{4} $& 43.\  $x_1^{3}x_2x_3^{4}x_4^{2} $& 44.\  $x_1^{3}x_2x_3^{6} $& 45.\  $x_1^{3}x_2^{3}x_4^{4}$\cr 
46.\  $x_1^{3}x_2^{3}x_3^{4} $& 47.\  $x_1^{3}x_2^{4}x_3x_4^{2} $& 48.\  $x_1^{3}x_2^{5}x_4^{2} $& 49.\  $x_1^{3}x_2^{5}x_3^{2} $& 50.\  $x_1^{3}x_2^{7}$\cr 
51.\  $x_1^{7}x_4^{3} $& 52.\  $x_1^{7}x_3x_4^{2} $& 53.\  $x_1^{7}x_3^{3} $& 54.\  $x_1^{7}x_2x_4^{2} $& 55.\  $x_1^{7}x_2x_3^{2}$\cr 
56.\  $x_1^{7}x_2^{3} $& 57.\  $x_1x_2x_3x_4^{7} $& 58.\  $x_1x_2x_3^{3}x_4^{5} $& 59.\  $x_1x_2x_3^{7}x_4 $& 60.\  $x_1x_2^{3}x_3x_4^{5}$\cr 
61.\  $x_1x_2^{3}x_3^{5}x_4 $& 62.\  $x_1x_2^{7}x_3x_4 $& 63.\  $x_1^{3}x_2x_3x_4^{5} $& 64.\  $x_1^{3}x_2x_3^{5}x_4 $& 65.\  $x_1^{3}x_2^{5}x_3x_4$\cr 
66.\  $x_1^{7}x_2x_3x_4 $& 67.\  $x_1x_2^{3}x_3^{3}x_4^{3} $& 68.\  $x_1^{3}x_2x_3^{3}x_4^{3} $& 69.\  $x_1^{3}x_2^{3}x_3x_4^{3} $& 70.\  $x_1^{3}x_2^{3}x_3^{3}x_4$\cr
\end{tabular}}

\medskip
\subsubsection{The admissible monomials of degree 10 in ${P_5}$.}\

\medskip 
$B_5(10) = B_5^0(10) \cup B_5^+(10)$, where $B_5^0(10) = \Phi^0(B_4(10))$, $|B_5^0(10)| =230$ and 
$$B_5^+(10) = B_5^+(2,2,1)\cup B_5^+(2,4)\cup B_5(4,1,1)\cup B_5(4,3).$$

\medskip
$B_5^+(2,2,1)$ is the set of 5 monomials:
$$ x_1x_2x_3^{2}x_4^{2}x_5^{4}\quad    x_1x_2x_3^{2}x_4^{4}x_5^{2} \quad    x_1x_2^{2}x_3x_4^{2}x_5^{4} \quad x_1x_2^{2}x_3x_4^{4}x_5^{2} \quad  x_1x_2^{2}x_3^{4}x_4x_5^{2}.$$

\medskip
$B_5^+(2,4)$ is the set of 5 monomials:
$$x_1x_2^{2}x_3^{2}x_4^{2}x_5^{3}\quad x_1x_2^{2}x_3^{2}x_4^{3}x_5^{2} \quad x_1x_2^{2}x_3^{3}x_4^{2}x_5^{2} \quad x_1x_2^{3}x_3^{2}x_4^{2}x_5^{2} \quad x_1^{3}x_2x_3^{2}x_4^{2}x_5^{2}. $$

\bigskip
$B_5^+(4,1,1)$ is the set of 20 monomials:

\bigskip
 \centerline{\begin{tabular}{llll}
1.\  $x_1x_2x_3x_4x_5^{6} $& 2.\  $x_1x_2x_3x_4^{2}x_5^{5} $& 3.\  $x_1x_2x_3x_4^{3}x_5^{4} $& 4.\  $x_1x_2x_3x_4^{6}x_5$\cr 
5.\  $x_1x_2x_3^{2}x_4x_5^{5} $& 6.\  $x_1x_2x_3^{2}x_4^{5}x_5 $& 7.\  $x_1x_2x_3^{3}x_4x_5^{4} $& 8.\  $x_1x_2x_3^{3}x_4^{4}x_5$\cr 
9.\  $x_1x_2x_3^{6}x_4x_5 $& 10.\  $x_1x_2^{2}x_3x_4x_5^{5} $& 11.\  $x_1x_2^{2}x_3x_4^{5}x_5 $& 12.\  $x_1x_2^{2}x_3^{5}x_4x_5$\cr 
13.\  $x_1x_2^{3}x_3x_4x_5^{4} $& 14.\  $x_1x_2^{3}x_3x_4^{4}x_5 $& 15.\  $x_1x_2^{3}x_3^{4}x_4x_5 $& 16.\  $x_1x_2^{6}x_3x_4x_5$\cr 
17.\  $x_1^{3}x_2x_3x_4x_5^{4} $& 18.\  $x_1^{3}x_2x_3x_4^{4}x_5 $& 19.\  $x_1^{3}x_2x_3^{4}x_4x_5 $& 20.\  $x_1^{3}x_2^{4}x_3x_4x_5$\cr 
\end{tabular}}

\medskip
$B_5^+(4,3)$ is the set of 20 monomials:

\medskip
 \centerline{\begin{tabular}{llll}
1.\  $x_1x_2x_3^{2}x_4^{3}x_5^{3} $& 2.\  $x_1x_2x_3^{3}x_4^{2}x_5^{3} $& 3.\  $x_1x_2x_3^{3}x_4^{3}x_5^{2} $& 4.\  $x_1x_2^{2}x_3x_4^{3}x_5^{3}$\cr 
5.\  $x_1x_2^{2}x_3^{3}x_4x_5^{3} $& 6.\  $x_1x_2^{2}x_3^{3}x_4^{3}x_5 $& 7.\  $x_1x_2^{3}x_3x_4^{2}x_5^{3} $& 8.\  $x_1x_2^{3}x_3x_4^{3}x_5^{2}$\cr 
9.\  $x_1x_2^{3}x_3^{2}x_4x_5^{3} $& 10.\  $x_1x_2^{3}x_3^{2}x_4^{3}x_5 $& 11.\  $x_1x_2^{3}x_3^{3}x_4x_5^{2} $& 12.\  $x_1x_2^{3}x_3^{3}x_4^{2}x_5$\cr 
13.\  $x_1^{3}x_2x_3x_4^{2}x_5^{3} $& 14.\  $x_1^{3}x_2x_3x_4^{3}x_5^{2} $& 15.\  $x_1^{3}x_2x_3^{2}x_4x_5^{3} $& 16.\  $x_1^{3}x_2x_3^{2}x_4^{3}x_5$\cr 
17.\  $x_1^{3}x_2x_3^{3}x_4x_5^{2} $& 18.\  $x_1^{3}x_2x_3^{3}x_4^{2}x_5 $& 19.\  $x_1^{3}x_2^{3}x_3x_4x_5^{2} $& 20.\  $x_1^{3}x_2^{3}x_3x_4^{2}x_5$\cr 
\end{tabular}}

\medskip
\subsection{The admissible monomials of degree 23 in ${P_5}$.}\ \label{s73}

\medskip
\subsubsection{The admissible monomials of degree 23 in ${P_4}$.}\ 

\medskip 
$B_4(23)$ is the set of 155 monomials $a_t = a_{23,t}, \ 1 \leqslant t \leqslant 155$:

\medskip
 \centerline{\begin{tabular}{llll}
$1.\ \   x_1x_2x_3^{7}x_4^{14}$& $2.\ \   x_1x_2x_3^{14}x_4^{7}$& $3.\ \   x_1x_2^{2}x_3^{7}x_4^{13}$& $4.\ \   x_1x_2^{2}x_3^{13}x_4^{7}$\cr 
$5.\ \   x_1x_2^{3}x_3^{5}x_4^{14}$& $6.\ \   x_1x_2^{3}x_3^{6}x_4^{13}$& $7.\ \   x_1x_2^{3}x_3^{7}x_4^{12}$& $8.\ \   x_1x_2^{3}x_3^{12}x_4^{7}$\cr 
$9.\ \   x_1x_2^{3}x_3^{13}x_4^{6}$& $10.\ \   x_1x_2^{3}x_3^{14}x_4^{5}$& $11.\ \   x_1x_2^{6}x_3^{3}x_4^{13}$& $12.\ \   x_1x_2^{6}x_3^{7}x_4^{9}$\cr 
$13.\ \   x_1x_2^{6}x_3^{11}x_4^{5}$& $14.\ \   x_1x_2^{7}x_3x_4^{14}$& $15.\ \   x_1x_2^{7}x_3^{2}x_4^{13}$& $16.\ \   x_1x_2^{7}x_3^{3}x_4^{12}$\cr 
$17.\ \   x_1x_2^{7}x_3^{6}x_4^{9}$& $18.\ \   x_1x_2^{7}x_3^{7}x_4^{8}$& $19.\ \   x_1x_2^{7}x_3^{10}x_4^{5}$& $20.\ \   x_1x_2^{7}x_3^{11}x_4^{4}$\cr 
$21.\ \   x_1x_2^{7}x_3^{14}x_4$& $22.\ \   x_1x_2^{14}x_3x_4^{7}$& $23.\ \   x_1x_2^{14}x_3^{3}x_4^{5}$& $24.\ \   x_1x_2^{14}x_3^{7}x_4$\cr 
$25.\ \   x_1^{3}x_2x_3^{5}x_4^{14}$& $26.\ \   x_1^{3}x_2x_3^{6}x_4^{13}$& $27.\ \   x_1^{3}x_2x_3^{7}x_4^{12}$& $28.\ \   x_1^{3}x_2x_3^{12}x_4^{7}$\cr 
$29.\ \   x_1^{3}x_2x_3^{13}x_4^{6}$& $30.\ \   x_1^{3}x_2x_3^{14}x_4^{5}$& $31.\ \   x_1^{3}x_2^{3}x_3^{4}x_4^{13}$& $32.\ \   x_1^{3}x_2^{3}x_3^{5}x_4^{12}$\cr 
$33.\ \   x_1^{3}x_2^{3}x_3^{12}x_4^{5}$& $34.\ \   x_1^{3}x_2^{3}x_3^{13}x_4^{4}$& $35.\ \   x_1^{3}x_2^{4}x_3^{3}x_4^{13}$& $36.\ \   x_1^{3}x_2^{4}x_3^{7}x_4^{9}$\cr 
$37.\ \   x_1^{3}x_2^{4}x_3^{11}x_4^{5}$& $38.\ \   x_1^{3}x_2^{5}x_3x_4^{14}$& $39.\ \   x_1^{3}x_2^{5}x_3^{2}x_4^{13}$& $40.\ \   x_1^{3}x_2^{5}x_3^{3}x_4^{12}$\cr 
$41.\ \   x_1^{3}x_2^{5}x_3^{6}x_4^{9}$& $42.\ \   x_1^{3}x_2^{5}x_3^{7}x_4^{8}$& $43.\ \   x_1^{3}x_2^{5}x_3^{10}x_4^{5}$& $44.\ \   x_1^{3}x_2^{5}x_3^{11}x_4^{4}$\cr 
$45.\ \   x_1^{3}x_2^{5}x_3^{14}x_4$& $46.\ \   x_1^{3}x_2^{7}x_3x_4^{12}$& $47.\ \   x_1^{3}x_2^{7}x_3^{4}x_4^{9}$& $48.\ \   x_1^{3}x_2^{7}x_3^{5}x_4^{8}$\cr 
$49.\ \   x_1^{3}x_2^{7}x_3^{8}x_4^{5}$& $50.\ \   x_1^{3}x_2^{7}x_3^{9}x_4^{4}$& $51.\ \   x_1^{3}x_2^{7}x_3^{12}x_4$& $52.\ \   x_1^{3}x_2^{13}x_3x_4^{6}$\cr 
$53.\ \   x_1^{3}x_2^{13}x_3^{2}x_4^{5}$& $54.\ \   x_1^{3}x_2^{13}x_3^{3}x_4^{4}$& $55.\ \   x_1^{3}x_2^{13}x_3^{6}x_4$& $56.\ \   x_1^{7}x_2x_3x_4^{14}$\cr 
$57.\ \   x_1^{7}x_2x_3^{2}x_4^{13}$& $58.\ \   x_1^{7}x_2x_3^{3}x_4^{12}$& $59.\ \   x_1^{7}x_2x_3^{6}x_4^{9}$& $60.\ \   x_1^{7}x_2x_3^{7}x_4^{8}$\cr 
$61.\ \   x_1^{7}x_2x_3^{10}x_4^{5}$& $62.\ \   x_1^{7}x_2x_3^{11}x_4^{4}$& $63.\ \   x_1^{7}x_2x_3^{14}x_4$& $64.\ \   x_1^{7}x_2^{3}x_3x_4^{12}$\cr 
$65.\ \   x_1^{7}x_2^{3}x_3^{4}x_4^{9}$& $66.\ \   x_1^{7}x_2^{3}x_3^{5}x_4^{8}$& $67.\ \   x_1^{7}x_2^{3}x_3^{8}x_4^{5}$& $68.\ \   x_1^{7}x_2^{3}x_3^{9}x_4^{4}$\cr 
$69.\ \   x_1^{7}x_2^{3}x_3^{12}x_4$& $70.\ \   x_1^{7}x_2^{7}x_3x_4^{8}$& $71.\ \   x_1^{7}x_2^{7}x_3^{8}x_4$& $72.\ \   x_1^{7}x_2^{9}x_3^{2}x_4^{5}$\cr 
$73.\ \   x_1^{7}x_2^{9}x_3^{3}x_4^{4}$& $74.\ \   x_1^{7}x_2^{11}x_3x_4^{4}$& $75.\ \   x_1^{7}x_2^{11}x_3^{4}x_4$& $76.\ \   x_2x_3^{7}x_4^{15}$\cr 
$77.\ \   x_2x_3^{15}x_4^{7}$& $78.\ \   x_2^{3}x_3^{5}x_4^{15}$& $79.\ \   x_2^{3}x_3^{7}x_4^{13}$& $80.\ \   x_2^{3}x_3^{13}x_4^{7}$\cr 
$81.\ \   x_2^{3}x_3^{15}x_4^{5}$& $82.\ \   x_2^{7}x_3x_4^{15}$& $83.\ \   x_2^{7}x_3^{3}x_4^{13}$& $84.\ \   x_2^{7}x_3^{7}x_4^{9}$\cr 
$85.\ \   x_2^{7}x_3^{11}x_4^{5}$& $86.\ \   x_2^{7}x_3^{15}x_4$& $87.\ \   x_2^{15}x_3x_4^{7}$& $88.\ \   x_2^{15}x_3^{3}x_4^{5}$\cr 
$89.\ \   x_2^{15}x_3^{7}x_4$& $90.\ \   x_1x_3^{7}x_4^{15}$& $91.\ \   x_1x_3^{15}x_4^{7}$& $92.\ \   x_1x_2x_3^{6}x_4^{15}$\cr 
$93.\ \   x_1x_2x_3^{15}x_4^{6}$& $94.\ \   x_1x_2^{2}x_3^{5}x_4^{15}$& $95.\ \   x_1x_2^{2}x_3^{15}x_4^{5}$& $96.\ \   x_1x_2^{3}x_3^{4}x_4^{15}$\cr 
$97.\ \   x_1x_2^{3}x_3^{15}x_4^{4}$& $98.\ \   x_1x_2^{6}x_3x_4^{15}$& $99.\ \   x_1x_2^{6}x_3^{15}x_4$& $100.\ \   x_1x_2^{7}x_4^{15}$\cr 
$101.\ \   x_1x_2^{7}x_3^{15}$& $102.\ \   x_1x_2^{15}x_4^{7}$& $103.\ \   x_1x_2^{15}x_3x_4^{6}$& $104.\ \   x_1x_2^{15}x_3^{2}x_4^{5}$\cr 
$105.\ \   x_1x_2^{15}x_3^{3}x_4^{4}$& $106.\ \   x_1x_2^{15}x_3^{6}x_4$& $107.\ \   x_1x_2^{15}x_3^{7}$& $108.\ \   x_1^{3}x_3^{5}x_4^{15}$\cr 
$109.\ \   x_1^{3}x_3^{7}x_4^{13}$& $110.\ \   x_1^{3}x_3^{13}x_4^{7}$& $111.\ \   x_1^{3}x_3^{15}x_4^{5}$& $112.\ \   x_1^{3}x_2x_3^{4}x_4^{15}$\cr 
$113.\ \   x_1^{3}x_2x_3^{15}x_4^{4}$& $114.\ \   x_1^{3}x_2^{4}x_3x_4^{15}$& $115.\ \   x_1^{3}x_2^{4}x_3^{15}x_4$& $116.\ \   x_1^{3}x_2^{5}x_4^{15}$\cr 
$117.\ \   x_1^{3}x_2^{5}x_3^{15}$& $118.\ \   x_1^{3}x_2^{7}x_4^{13}$& $119.\ \   x_1^{3}x_2^{7}x_3^{13}$& $120.\ \   x_1^{3}x_2^{13}x_4^{7}$\cr 
$121.\ \   x_1^{3}x_2^{13}x_3^{7}$& $122.\ \   x_1^{3}x_2^{15}x_4^{5}$& $123.\ \   x_1^{3}x_2^{15}x_3x_4^{4}$& $124.\ \   x_1^{3}x_2^{15}x_3^{4}x_4$\cr 
$125.\ \   x_1^{3}x_2^{15}x_3^{5}$& $126.\ \   x_1^{7}x_3x_4^{15}$& $127.\ \   x_1^{7}x_3^{3}x_4^{13}$& $128.\ \   x_1^{7}x_3^{7}x_4^{9}$\cr 
$129.\ \   x_1^{7}x_3^{11}x_4^{5}$& $130.\ \   x_1^{7}x_3^{15}x_4$& $131.\ \   x_1^{7}x_2x_4^{15}$& $132.\ \   x_1^{7}x_2x_3^{15}$\cr 
$133.\ \   x_1^{7}x_2^{3}x_4^{13}$& $134.\ \   x_1^{7}x_2^{3}x_3^{13}$& $135.\ \   x_1^{7}x_2^{7}x_4^{9}$& $136.\ \   x_1^{7}x_2^{7}x_3^{9}$\cr 
$137.\ \   x_1^{7}x_2^{11}x_4^{5}$& $138.\ \   x_1^{7}x_2^{11}x_3^{5}$& $139.\ \   x_1^{7}x_2^{15}x_4$& $140.\ \   x_1^{7}x_2^{15}x_3$\cr 
$141.\ \   x_1^{15}x_3x_4^{7}$& $142.\ \   x_1^{15}x_3^{3}x_4^{5}$& $143.\ \   x_1^{15}x_3^{7}x_4$& $144.\ \   x_1^{15}x_2x_4^{7}$\cr 
$145.\ \   x_1^{15}x_2x_3x_4^{6}$& $146.\ \   x_1^{15}x_2x_3^{2}x_4^{5}$& $147.\ \   x_1^{15}x_2x_3^{3}x_4^{4}$& $148.\ \   x_1^{15}x_2x_3^{6}x_4$\cr 
\end{tabular}}

 \centerline{\begin{tabular}{llll}
$149.\ \   x_1^{15}x_2x_3^{7}$& $150.\ \   x_1^{15}x_2^{3}x_4^{5}$& $151.\ \   x_1^{15}x_2^{3}x_3x_4^{4}$& $152.\ \   x_1^{15}x_2^{3}x_3^{4}x_4$\cr 
$153.\ \   x_1^{15}x_2^{3}x_3^{5}$& $154.\ \   x_1^{15}x_2^{7}x_4$& $155.\ \   x_1^{15}x_2^{7}x_3$.&
\end{tabular}}

\medskip
\subsubsection{The admissible monomials of degree 23 in ${P_5}$.}\ 

\medskip 
We have $B_5(23) = B_5^0(23) \cup \psi(B_5(9)) \cup \left(B_5^+(23)\cap \text{Ker}(\widetilde{Sq}^0_*)_{(5,9)}\right)$, where $B_5^0(23)$ $ = \Phi^0(B_4(23))$, $|B_5^0(35)| =635$, $|\psi(B_5(9))| = 191$ with $\psi: P_5 \to P_5, \ \psi(x) = X_\emptyset x^2$, and 
$$B_5^+(23)\cap \text{Ker}(\widetilde{Sq}^0_*)_{(5,9)} = B_5^+(3,2,2,1) \cup B_5^+(3,4,1,1) \cup B_5^+(3,4,3).$$ 

\medskip
$B_5^+(3,2,2,1)$ is the set of 290 monomials $b_t = b_{23,t}, \ 1 \leqslant t \leqslant 290$:

\medskip
 \centerline{\begin{tabular}{rlrlrl}
1. & $x_1x_2x_3x_4^{6}x_5^{14} $& 2. & $x_1x_2x_3x_4^{14}x_5^{6} $& 3. & $x_1x_2x_3^{2}x_4^{4}x_5^{15} $\cr  
4. & $x_1x_2x_3^{2}x_4^{5}x_5^{14} $& 5. & $x_1x_2x_3^{2}x_4^{6}x_5^{13} $& 6. & $x_1x_2x_3^{2}x_4^{7}x_5^{12} $\cr  
7. & $x_1x_2x_3^{2}x_4^{12}x_5^{7} $& 8. & $x_1x_2x_3^{2}x_4^{13}x_5^{6} $& 9. & $x_1x_2x_3^{2}x_4^{14}x_5^{5} $\cr  
10. & $x_1x_2x_3^{2}x_4^{15}x_5^{4} $& 11. & $x_1x_2x_3^{3}x_4^{4}x_5^{14} $& 12. & $x_1x_2x_3^{3}x_4^{6}x_5^{12} $\cr  
13. & $x_1x_2x_3^{3}x_4^{12}x_5^{6} $& 14. & $x_1x_2x_3^{3}x_4^{14}x_5^{4} $& 15. & $x_1x_2x_3^{6}x_4x_5^{14} $\cr  
16. & $x_1x_2x_3^{6}x_4^{2}x_5^{13} $& 17. & $x_1x_2x_3^{6}x_4^{3}x_5^{12} $& 18. & $x_1x_2x_3^{6}x_4^{6}x_5^{9} $\cr  
19. & $x_1x_2x_3^{6}x_4^{7}x_5^{8} $& 20. & $x_1x_2x_3^{6}x_4^{10}x_5^{5} $& 21. & $x_1x_2x_3^{6}x_4^{11}x_5^{4} $\cr  
22. & $x_1x_2x_3^{6}x_4^{14}x_5 $& 23. & $x_1x_2x_3^{7}x_4^{2}x_5^{12} $& 24. & $x_1x_2x_3^{7}x_4^{6}x_5^{8} $\cr  
25. & $x_1x_2x_3^{7}x_4^{10}x_5^{4} $& 26. & $x_1x_2x_3^{14}x_4x_5^{6} $& 27. & $x_1x_2x_3^{14}x_4^{2}x_5^{5} $\cr  
28. & $x_1x_2x_3^{14}x_4^{3}x_5^{4} $& 29. & $x_1x_2x_3^{14}x_4^{6}x_5 $& 30. & $x_1x_2x_3^{15}x_4^{2}x_5^{4} $\cr  
31. & $x_1x_2^{2}x_3x_4^{4}x_5^{15} $& 32. & $x_1x_2^{2}x_3x_4^{5}x_5^{14} $& 33. & $x_1x_2^{2}x_3x_4^{6}x_5^{13} $\cr  
34. & $x_1x_2^{2}x_3x_4^{7}x_5^{12} $& 35. & $x_1x_2^{2}x_3x_4^{12}x_5^{7} $& 36. & $x_1x_2^{2}x_3x_4^{13}x_5^{6} $\cr  
37. & $x_1x_2^{2}x_3x_4^{14}x_5^{5} $& 38. & $x_1x_2^{2}x_3x_4^{15}x_5^{4} $& 39. & $x_1x_2^{2}x_3^{3}x_4^{4}x_5^{13} $\cr  
40. & $x_1x_2^{2}x_3^{3}x_4^{5}x_5^{12} $& 41. & $x_1x_2^{2}x_3^{3}x_4^{12}x_5^{5} $& 42. & $x_1x_2^{2}x_3^{3}x_4^{13}x_5^{4} $\cr  
43. & $x_1x_2^{2}x_3^{4}x_4x_5^{15} $& 44. & $x_1x_2^{2}x_3^{4}x_4^{3}x_5^{13} $& 45. & $x_1x_2^{2}x_3^{4}x_4^{7}x_5^{9} $\cr  
46. & $x_1x_2^{2}x_3^{4}x_4^{9}x_5^{7} $& 47. & $x_1x_2^{2}x_3^{4}x_4^{11}x_5^{5} $& 48. & $x_1x_2^{2}x_3^{4}x_4^{15}x_5 $\cr  
49. & $x_1x_2^{2}x_3^{5}x_4x_5^{14} $& 50. & $x_1x_2^{2}x_3^{5}x_4^{2}x_5^{13} $& 51. & $x_1x_2^{2}x_3^{5}x_4^{3}x_5^{12} $\cr  
52. & $x_1x_2^{2}x_3^{5}x_4^{6}x_5^{9} $& 53. & $x_1x_2^{2}x_3^{5}x_4^{7}x_5^{8} $& 54. & $x_1x_2^{2}x_3^{5}x_4^{8}x_5^{7} $\cr  
55. & $x_1x_2^{2}x_3^{5}x_4^{9}x_5^{6} $& 56. & $x_1x_2^{2}x_3^{5}x_4^{10}x_5^{5} $& 57. & $x_1x_2^{2}x_3^{5}x_4^{11}x_5^{4} $\cr  
58. & $x_1x_2^{2}x_3^{5}x_4^{14}x_5 $& 59. & $x_1x_2^{2}x_3^{7}x_4x_5^{12} $& 60. & $x_1x_2^{2}x_3^{7}x_4^{4}x_5^{9} $\cr  
61. & $x_1x_2^{2}x_3^{7}x_4^{5}x_5^{8} $& 62. & $x_1x_2^{2}x_3^{7}x_4^{8}x_5^{5} $& 63. & $x_1x_2^{2}x_3^{7}x_4^{9}x_5^{4} $\cr  
64. & $x_1x_2^{2}x_3^{7}x_4^{12}x_5 $& 65. & $x_1x_2^{2}x_3^{12}x_4x_5^{7} $& 66. & $x_1x_2^{2}x_3^{12}x_4^{3}x_5^{5} $\cr  
67. & $x_1x_2^{2}x_3^{12}x_4^{7}x_5 $& 68. & $x_1x_2^{2}x_3^{13}x_4x_5^{6} $& 69. & $x_1x_2^{2}x_3^{13}x_4^{2}x_5^{5} $\cr  
70. & $x_1x_2^{2}x_3^{13}x_4^{3}x_5^{4} $& 71. & $x_1x_2^{2}x_3^{13}x_4^{6}x_5 $& 72. & $x_1x_2^{2}x_3^{15}x_4x_5^{4} $\cr  
73. & $x_1x_2^{2}x_3^{15}x_4^{4}x_5 $& 74. & $x_1x_2^{3}x_3x_4^{4}x_5^{14} $& 75. & $x_1x_2^{3}x_3x_4^{6}x_5^{12} $\cr  
76. & $x_1x_2^{3}x_3x_4^{12}x_5^{6} $& 77. & $x_1x_2^{3}x_3x_4^{14}x_5^{4} $& 78. & $x_1x_2^{3}x_3^{2}x_4^{4}x_5^{13} $\cr  
79. & $x_1x_2^{3}x_3^{2}x_4^{5}x_5^{12} $& 80. & $x_1x_2^{3}x_3^{2}x_4^{12}x_5^{5} $& 81. & $x_1x_2^{3}x_3^{2}x_4^{13}x_5^{4} $\cr  
82. & $x_1x_2^{3}x_3^{3}x_4^{4}x_5^{12} $& 83. & $x_1x_2^{3}x_3^{3}x_4^{12}x_5^{4} $& 84. & $x_1x_2^{3}x_3^{4}x_4x_5^{14} $\cr  
85. & $x_1x_2^{3}x_3^{4}x_4^{2}x_5^{13} $& 86. & $x_1x_2^{3}x_3^{4}x_4^{3}x_5^{12} $& 87. & $x_1x_2^{3}x_3^{4}x_4^{6}x_5^{9} $\cr  
88. & $x_1x_2^{3}x_3^{4}x_4^{7}x_5^{8} $& 89. & $x_1x_2^{3}x_3^{4}x_4^{8}x_5^{7} $& 90. & $x_1x_2^{3}x_3^{4}x_4^{9}x_5^{6} $\cr  
91. & $x_1x_2^{3}x_3^{4}x_4^{10}x_5^{5} $& 92. & $x_1x_2^{3}x_3^{4}x_4^{11}x_5^{4} $& 93. & $x_1x_2^{3}x_3^{4}x_4^{14}x_5 $\cr  
94. & $x_1x_2^{3}x_3^{5}x_4^{2}x_5^{12} $& 95. & $x_1x_2^{3}x_3^{5}x_4^{6}x_5^{8} $& 96. & $x_1x_2^{3}x_3^{5}x_4^{8}x_5^{6} $\cr  
97. & $x_1x_2^{3}x_3^{5}x_4^{10}x_5^{4} $& 98. & $x_1x_2^{3}x_3^{6}x_4x_5^{12} $& 99. & $x_1x_2^{3}x_3^{6}x_4^{4}x_5^{9} $\cr  
100. & $x_1x_2^{3}x_3^{6}x_4^{5}x_5^{8} $& 101. & $x_1x_2^{3}x_3^{6}x_4^{8}x_5^{5} $& 102. & $x_1x_2^{3}x_3^{6}x_4^{9}x_5^{4} $\cr  
103. & $x_1x_2^{3}x_3^{6}x_4^{12}x_5 $& 104. & $x_1x_2^{3}x_3^{7}x_4^{4}x_5^{8} $& 105. & $x_1x_2^{3}x_3^{7}x_4^{8}x_5^{4} $\cr  
106. & $x_1x_2^{3}x_3^{12}x_4x_5^{6} $& 107. & $x_1x_2^{3}x_3^{12}x_4^{2}x_5^{5} $& 108. & $x_1x_2^{3}x_3^{12}x_4^{3}x_5^{4} $\cr  
109. & $x_1x_2^{3}x_3^{12}x_4^{6}x_5 $& 110. & $x_1x_2^{3}x_3^{13}x_4^{2}x_5^{4} $& 111. & $x_1x_2^{3}x_3^{14}x_4x_5^{4} $\cr  
 \end{tabular}}

\centerline{\begin{tabular}{rlrlrl}
112. & $x_1x_2^{3}x_3^{14}x_4^{4}x_5 $& 113. & $x_1x_2^{6}x_3x_4x_5^{14} $& 114. & $x_1x_2^{6}x_3x_4^{2}x_5^{13} $\cr  
115. & $x_1x_2^{6}x_3x_4^{3}x_5^{12} $& 116. & $x_1x_2^{6}x_3x_4^{6}x_5^{9} $& 117. & $x_1x_2^{6}x_3x_4^{7}x_5^{8} $\cr  
118. & $x_1x_2^{6}x_3x_4^{10}x_5^{5} $& 119. & $x_1x_2^{6}x_3x_4^{11}x_5^{4} $& 120. & $x_1x_2^{6}x_3x_4^{14}x_5 $\cr  
121. & $x_1x_2^{6}x_3^{3}x_4x_5^{12} $& 122. & $x_1x_2^{6}x_3^{3}x_4^{4}x_5^{9} $& 123. & $x_1x_2^{6}x_3^{3}x_4^{5}x_5^{8} $\cr  
124. & $x_1x_2^{6}x_3^{3}x_4^{8}x_5^{5} $& 125. & $x_1x_2^{6}x_3^{3}x_4^{9}x_5^{4} $& 126. & $x_1x_2^{6}x_3^{3}x_4^{12}x_5 $\cr  
127. & $x_1x_2^{6}x_3^{7}x_4x_5^{8} $& 128. & $x_1x_2^{6}x_3^{7}x_4^{8}x_5 $& 129. & $x_1x_2^{6}x_3^{9}x_4^{2}x_5^{5} $\cr  
130. & $x_1x_2^{6}x_3^{9}x_4^{3}x_5^{4} $& 131. & $x_1x_2^{6}x_3^{11}x_4x_5^{4} $& 132. & $x_1x_2^{6}x_3^{11}x_4^{4}x_5 $\cr  
133. & $x_1x_2^{7}x_3x_4^{2}x_5^{12} $& 134. & $x_1x_2^{7}x_3x_4^{6}x_5^{8} $& 135. & $x_1x_2^{7}x_3x_4^{10}x_5^{4} $\cr  
136. & $x_1x_2^{7}x_3^{2}x_4x_5^{12} $& 137. & $x_1x_2^{7}x_3^{2}x_4^{4}x_5^{9} $& 138. & $x_1x_2^{7}x_3^{2}x_4^{5}x_5^{8} $\cr  
139. & $x_1x_2^{7}x_3^{2}x_4^{8}x_5^{5} $& 140. & $x_1x_2^{7}x_3^{2}x_4^{9}x_5^{4} $& 141. & $x_1x_2^{7}x_3^{2}x_4^{12}x_5 $\cr  
142. & $x_1x_2^{7}x_3^{3}x_4^{4}x_5^{8} $& 143. & $x_1x_2^{7}x_3^{3}x_4^{8}x_5^{4} $& 144. & $x_1x_2^{7}x_3^{6}x_4x_5^{8} $\cr  
145. & $x_1x_2^{7}x_3^{6}x_4^{8}x_5 $& 146. & $x_1x_2^{7}x_3^{8}x_4^{2}x_5^{5} $& 147. & $x_1x_2^{7}x_3^{8}x_4^{3}x_5^{4} $\cr  
148. & $x_1x_2^{7}x_3^{9}x_4^{2}x_5^{4} $& 149. & $x_1x_2^{7}x_3^{10}x_4x_5^{4} $& 150. & $x_1x_2^{7}x_3^{10}x_4^{4}x_5 $\cr  
151. & $x_1x_2^{14}x_3x_4x_5^{6} $& 152. & $x_1x_2^{14}x_3x_4^{2}x_5^{5} $& 153. & $x_1x_2^{14}x_3x_4^{3}x_5^{4} $\cr  
154. & $x_1x_2^{14}x_3x_4^{6}x_5 $& 155. & $x_1x_2^{14}x_3^{3}x_4x_5^{4} $& 156. & $x_1x_2^{14}x_3^{3}x_4^{4}x_5 $\cr  
157. & $x_1x_2^{15}x_3x_4^{2}x_5^{4} $& 158. & $x_1x_2^{15}x_3^{2}x_4x_5^{4} $& 159. & $x_1x_2^{15}x_3^{2}x_4^{4}x_5 $\cr  
160. & $x_1^{3}x_2x_3x_4^{4}x_5^{14} $& 161. & $x_1^{3}x_2x_3x_4^{6}x_5^{12} $& 162. & $x_1^{3}x_2x_3x_4^{12}x_5^{6} $\cr  
163. & $x_1^{3}x_2x_3x_4^{14}x_5^{4} $& 164. & $x_1^{3}x_2x_3^{2}x_4^{4}x_5^{13} $& 165. & $x_1^{3}x_2x_3^{2}x_4^{5}x_5^{12} $\cr  
166. & $x_1^{3}x_2x_3^{2}x_4^{12}x_5^{5} $& 167. & $x_1^{3}x_2x_3^{2}x_4^{13}x_5^{4} $& 168. & $x_1^{3}x_2x_3^{3}x_4^{4}x_5^{12} $\cr  
169. & $x_1^{3}x_2x_3^{3}x_4^{12}x_5^{4} $& 170. & $x_1^{3}x_2x_3^{4}x_4x_5^{14} $& 171. & $x_1^{3}x_2x_3^{4}x_4^{2}x_5^{13} $\cr  
172. & $x_1^{3}x_2x_3^{4}x_4^{3}x_5^{12} $& 173. & $x_1^{3}x_2x_3^{4}x_4^{6}x_5^{9} $& 174. & $x_1^{3}x_2x_3^{4}x_4^{7}x_5^{8} $\cr  
175. & $x_1^{3}x_2x_3^{4}x_4^{8}x_5^{7} $& 176. & $x_1^{3}x_2x_3^{4}x_4^{9}x_5^{6} $& 177. & $x_1^{3}x_2x_3^{4}x_4^{10}x_5^{5} $\cr  
178. & $x_1^{3}x_2x_3^{4}x_4^{11}x_5^{4} $& 179. & $x_1^{3}x_2x_3^{4}x_4^{14}x_5 $& 180. & $x_1^{3}x_2x_3^{5}x_4^{2}x_5^{12} $\cr  
181. & $x_1^{3}x_2x_3^{5}x_4^{6}x_5^{8} $& 182. & $x_1^{3}x_2x_3^{5}x_4^{8}x_5^{6} $& 183. & $x_1^{3}x_2x_3^{5}x_4^{10}x_5^{4} $\cr  
184. & $x_1^{3}x_2x_3^{6}x_4x_5^{12} $& 185. & $x_1^{3}x_2x_3^{6}x_4^{4}x_5^{9} $& 186. & $x_1^{3}x_2x_3^{6}x_4^{5}x_5^{8} $\cr  
187. & $x_1^{3}x_2x_3^{6}x_4^{8}x_5^{5} $& 188. & $x_1^{3}x_2x_3^{6}x_4^{9}x_5^{4} $& 189. & $x_1^{3}x_2x_3^{6}x_4^{12}x_5 $\cr  
190. & $x_1^{3}x_2x_3^{7}x_4^{4}x_5^{8} $& 191. & $x_1^{3}x_2x_3^{7}x_4^{8}x_5^{4} $& 192. & $x_1^{3}x_2x_3^{12}x_4x_5^{6} $\cr  
193. & $x_1^{3}x_2x_3^{12}x_4^{2}x_5^{5} $& 194. & $x_1^{3}x_2x_3^{12}x_4^{3}x_5^{4} $& 195. & $x_1^{3}x_2x_3^{12}x_4^{6}x_5 $\cr  
196. & $x_1^{3}x_2x_3^{13}x_4^{2}x_5^{4} $& 197. & $x_1^{3}x_2x_3^{14}x_4x_5^{4} $& 198. & $x_1^{3}x_2x_3^{14}x_4^{4}x_5 $\cr  
199. & $x_1^{3}x_2^{3}x_3x_4^{4}x_5^{12} $& 200. & $x_1^{3}x_2^{3}x_3x_4^{12}x_5^{4} $& 201. & $x_1^{3}x_2^{3}x_3^{4}x_4x_5^{12} $\cr  
202. & $x_1^{3}x_2^{3}x_3^{4}x_4^{4}x_5^{9} $& 203. & $x_1^{3}x_2^{3}x_3^{4}x_4^{5}x_5^{8} $& 204. & $x_1^{3}x_2^{3}x_3^{4}x_4^{8}x_5^{5} $\cr  
205. & $x_1^{3}x_2^{3}x_3^{4}x_4^{9}x_5^{4} $& 206. & $x_1^{3}x_2^{3}x_3^{4}x_4^{12}x_5 $& 207. & $x_1^{3}x_2^{3}x_3^{5}x_4^{4}x_5^{8} $\cr  
208. & $x_1^{3}x_2^{3}x_3^{5}x_4^{8}x_5^{4} $& 209. & $x_1^{3}x_2^{3}x_3^{12}x_4x_5^{4} $& 210. & $x_1^{3}x_2^{3}x_3^{12}x_4^{4}x_5 $\cr  
211. & $x_1^{3}x_2^{4}x_3x_4x_5^{14} $& 212. & $x_1^{3}x_2^{4}x_3x_4^{2}x_5^{13} $& 213. & $x_1^{3}x_2^{4}x_3x_4^{3}x_5^{12} $\cr  
214. & $x_1^{3}x_2^{4}x_3x_4^{6}x_5^{9} $& 215. & $x_1^{3}x_2^{4}x_3x_4^{7}x_5^{8} $& 216. & $x_1^{3}x_2^{4}x_3x_4^{10}x_5^{5} $\cr  
217. & $x_1^{3}x_2^{4}x_3x_4^{11}x_5^{4} $& 218. & $x_1^{3}x_2^{4}x_3x_4^{14}x_5 $& 219. & $x_1^{3}x_2^{4}x_3^{3}x_4x_5^{12} $\cr  
220. & $x_1^{3}x_2^{4}x_3^{3}x_4^{5}x_5^{8} $& 221. & $x_1^{3}x_2^{4}x_3^{3}x_4^{8}x_5^{5} $& 222. & $x_1^{3}x_2^{4}x_3^{3}x_4^{9}x_5^{4} $\cr  
223. & $x_1^{3}x_2^{4}x_3^{7}x_4x_5^{8} $& 224. & $x_1^{3}x_2^{4}x_3^{7}x_4^{8}x_5 $& 225. & $x_1^{3}x_2^{4}x_3^{9}x_4^{2}x_5^{5} $\cr  
226. & $x_1^{3}x_2^{4}x_3^{9}x_4^{3}x_5^{4} $& 227. & $x_1^{3}x_2^{4}x_3^{11}x_4x_5^{4} $& 228. & $x_1^{3}x_2^{5}x_3x_4^{2}x_5^{12} $\cr  
229. & $x_1^{3}x_2^{5}x_3x_4^{6}x_5^{8} $& 230. & $x_1^{3}x_2^{5}x_3x_4^{10}x_5^{4} $& 231. & $x_1^{3}x_2^{5}x_3^{2}x_4x_5^{12} $\cr  
232. & $x_1^{3}x_2^{5}x_3^{2}x_4^{4}x_5^{9} $& 233. & $x_1^{3}x_2^{5}x_3^{2}x_4^{5}x_5^{8} $& 234. & $x_1^{3}x_2^{5}x_3^{2}x_4^{8}x_5^{5} $\cr  
235. & $x_1^{3}x_2^{5}x_3^{2}x_4^{9}x_5^{4} $& 236. & $x_1^{3}x_2^{5}x_3^{2}x_4^{12}x_5 $& 237. & $x_1^{3}x_2^{5}x_3^{3}x_4^{4}x_5^{8} $\cr  
238. & $x_1^{3}x_2^{5}x_3^{3}x_4^{8}x_5^{4} $& 239. & $x_1^{3}x_2^{5}x_3^{6}x_4x_5^{8} $& 240. & $x_1^{3}x_2^{5}x_3^{6}x_4^{8}x_5 $\cr  
241. & $x_1^{3}x_2^{5}x_3^{8}x_4^{2}x_5^{5} $& 242. & $x_1^{3}x_2^{5}x_3^{8}x_4^{3}x_5^{4} $& 243. & $x_1^{3}x_2^{5}x_3^{9}x_4^{2}x_5^{4} $\cr  
244. & $x_1^{3}x_2^{5}x_3^{10}x_4x_5^{4} $& 245. & $x_1^{3}x_2^{5}x_3^{10}x_4^{4}x_5 $& 246. & $x_1^{3}x_2^{7}x_3x_4^{4}x_5^{8} $\cr  
247. & $x_1^{3}x_2^{7}x_3x_4^{8}x_5^{4} $& 248. & $x_1^{3}x_2^{7}x_3^{4}x_4x_5^{8} $& 249. & $x_1^{3}x_2^{7}x_3^{4}x_4^{8}x_5 $\cr  
250. & $x_1^{3}x_2^{7}x_3^{8}x_4x_5^{4} $& 251. & $x_1^{3}x_2^{7}x_3^{8}x_4^{4}x_5 $& 252. & $x_1^{3}x_2^{12}x_3x_4^{2}x_5^{5} $\cr  
253. & $x_1^{3}x_2^{12}x_3x_4^{3}x_5^{4} $& 254. & $x_1^{3}x_2^{12}x_3^{3}x_4x_5^{4} $& 255. & $x_1^{3}x_2^{13}x_3x_4^{2}x_5^{4} $\cr  
256. & $x_1^{3}x_2^{13}x_3^{2}x_4x_5^{4} $& 257. & $x_1^{3}x_2^{13}x_3^{2}x_4^{4}x_5 $& 258. & $x_1^{7}x_2x_3x_4^{2}x_5^{12} $\cr  
 \end{tabular}}
\centerline{\begin{tabular}{rlrlrl}
259. & $x_1^{7}x_2x_3x_4^{6}x_5^{8} $& 260. & $x_1^{7}x_2x_3x_4^{10}x_5^{4} $& 261. & $x_1^{7}x_2x_3^{2}x_4x_5^{12} $\cr  
262. & $x_1^{7}x_2x_3^{2}x_4^{4}x_5^{9} $& 263. & $x_1^{7}x_2x_3^{2}x_4^{5}x_5^{8} $& 264. & $x_1^{7}x_2x_3^{2}x_4^{8}x_5^{5} $\cr  
265. & $x_1^{7}x_2x_3^{2}x_4^{9}x_5^{4} $& 266. & $x_1^{7}x_2x_3^{2}x_4^{12}x_5 $& 267. & $x_1^{7}x_2x_3^{3}x_4^{4}x_5^{8} $\cr  
268. & $x_1^{7}x_2x_3^{3}x_4^{8}x_5^{4} $& 269. & $x_1^{7}x_2x_3^{6}x_4x_5^{8} $& 270. & $x_1^{7}x_2x_3^{6}x_4^{8}x_5 $\cr  
271. & $x_1^{7}x_2x_3^{8}x_4^{2}x_5^{5} $& 272. & $x_1^{7}x_2x_3^{8}x_4^{3}x_5^{4} $& 273. & $x_1^{7}x_2x_3^{9}x_4^{2}x_5^{4} $\cr  
274. & $x_1^{7}x_2x_3^{10}x_4x_5^{4} $& 275. & $x_1^{7}x_2x_3^{10}x_4^{4}x_5 $& 276. & $x_1^{7}x_2^{3}x_3x_4^{4}x_5^{8} $\cr  
277. & $x_1^{7}x_2^{3}x_3x_4^{8}x_5^{4} $& 278. & $x_1^{7}x_2^{3}x_3^{4}x_4x_5^{8} $& 279. & $x_1^{7}x_2^{3}x_3^{4}x_4^{8}x_5 $\cr  
280. & $x_1^{7}x_2^{3}x_3^{8}x_4x_5^{4} $& 281. & $x_1^{7}x_2^{3}x_3^{8}x_4^{4}x_5 $& 282. & $x_1^{7}x_2^{8}x_3x_4^{2}x_5^{5} $\cr  
283. & $x_1^{7}x_2^{8}x_3x_4^{3}x_5^{4} $& 284. & $x_1^{7}x_2^{8}x_3^{3}x_4x_5^{4} $& 285. & $x_1^{7}x_2^{9}x_3x_4^{2}x_5^{4} $\cr  
286. & $x_1^{7}x_2^{9}x_3^{2}x_4x_5^{4} $& 287. & $x_1^{7}x_2^{9}x_3^{2}x_4^{4}x_5 $& 288. & $x_1^{15}x_2x_3x_4^{2}x_5^{4} $\cr  
289. & $x_1^{15}x_2x_3^{2}x_4x_5^{4} $& 290. & $x_1^{15}x_2x_3^{2}x_4^{4}x_5 $& 
\end{tabular}}

\medskip
$B_5^+(3,4,1,1)$ is the set of 105 monomials $b_t = b_{23,t}, \ 291 \leqslant t \leqslant 395$:

\medskip
 \centerline{\begin{tabular}{rlrlrl}
291. & $x_1x_2^{2}x_3^{2}x_4^{3}x_5^{15} $& 292. & $x_1x_2^{2}x_3^{2}x_4^{7}x_5^{11} $& 293. & $x_1x_2^{2}x_3^{2}x_4^{15}x_5^{3} $\cr  
294. & $x_1x_2^{2}x_3^{3}x_4^{2}x_5^{15} $& 295. & $x_1x_2^{2}x_3^{3}x_4^{3}x_5^{14} $& 296. & $x_1x_2^{2}x_3^{3}x_4^{6}x_5^{11} $\cr  
297. & $x_1x_2^{2}x_3^{3}x_4^{7}x_5^{10} $& 298. & $x_1x_2^{2}x_3^{3}x_4^{14}x_5^{3} $& 299. & $x_1x_2^{2}x_3^{3}x_4^{15}x_5^{2} $\cr  
300. & $x_1x_2^{2}x_3^{7}x_4^{2}x_5^{11} $& 301. & $x_1x_2^{2}x_3^{7}x_4^{3}x_5^{10} $& 302. & $x_1x_2^{2}x_3^{7}x_4^{10}x_5^{3} $\cr  
303. & $x_1x_2^{2}x_3^{7}x_4^{11}x_5^{2} $& 304. & $x_1x_2^{2}x_3^{15}x_4^{2}x_5^{3} $& 305. & $x_1x_2^{2}x_3^{15}x_4^{3}x_5^{2} $\cr  
306. & $x_1x_2^{3}x_3^{2}x_4^{2}x_5^{15} $& 307. & $x_1x_2^{3}x_3^{2}x_4^{3}x_5^{14} $& 308. & $x_1x_2^{3}x_3^{2}x_4^{6}x_5^{11} $\cr  
309. & $x_1x_2^{3}x_3^{2}x_4^{7}x_5^{10} $& 310. & $x_1x_2^{3}x_3^{2}x_4^{14}x_5^{3} $& 311. & $x_1x_2^{3}x_3^{2}x_4^{15}x_5^{2} $\cr  
312. & $x_1x_2^{3}x_3^{3}x_4^{2}x_5^{14} $& 313. & $x_1x_2^{3}x_3^{3}x_4^{6}x_5^{10} $& 314. & $x_1x_2^{3}x_3^{3}x_4^{14}x_5^{2} $\cr  
315. & $x_1x_2^{3}x_3^{6}x_4^{2}x_5^{11} $& 316. & $x_1x_2^{3}x_3^{6}x_4^{3}x_5^{10} $& 317. & $x_1x_2^{3}x_3^{6}x_4^{10}x_5^{3} $\cr  
318. & $x_1x_2^{3}x_3^{6}x_4^{11}x_5^{2} $& 319. & $x_1x_2^{3}x_3^{7}x_4^{2}x_5^{10} $& 320. & $x_1x_2^{3}x_3^{7}x_4^{10}x_5^{2} $\cr  
321. & $x_1x_2^{3}x_3^{14}x_4^{2}x_5^{3} $& 322. & $x_1x_2^{3}x_3^{14}x_4^{3}x_5^{2} $& 323. & $x_1x_2^{3}x_3^{15}x_4^{2}x_5^{2} $\cr  
324. & $x_1x_2^{7}x_3^{2}x_4^{2}x_5^{11} $& 325. & $x_1x_2^{7}x_3^{2}x_4^{3}x_5^{10} $& 326. & $x_1x_2^{7}x_3^{2}x_4^{10}x_5^{3} $\cr  
327. & $x_1x_2^{7}x_3^{2}x_4^{11}x_5^{2} $& 328. & $x_1x_2^{7}x_3^{3}x_4^{2}x_5^{10} $& 329. & $x_1x_2^{7}x_3^{3}x_4^{10}x_5^{2} $\cr  
330. & $x_1x_2^{7}x_3^{10}x_4^{2}x_5^{3} $& 331. & $x_1x_2^{7}x_3^{10}x_4^{3}x_5^{2} $& 332. & $x_1x_2^{7}x_3^{11}x_4^{2}x_5^{2} $\cr  
333. & $x_1x_2^{15}x_3^{2}x_4^{2}x_5^{3} $& 334. & $x_1x_2^{15}x_3^{2}x_4^{3}x_5^{2} $& 335. & $x_1x_2^{15}x_3^{3}x_4^{2}x_5^{2} $\cr  
336. & $x_1^{3}x_2x_3^{2}x_4^{2}x_5^{15} $& 337. & $x_1^{3}x_2x_3^{2}x_4^{3}x_5^{14} $& 338. & $x_1^{3}x_2x_3^{2}x_4^{6}x_5^{11} $\cr  
339. & $x_1^{3}x_2x_3^{2}x_4^{7}x_5^{10} $& 340. & $x_1^{3}x_2x_3^{2}x_4^{14}x_5^{3} $& 341. & $x_1^{3}x_2x_3^{2}x_4^{15}x_5^{2} $\cr  
342. & $x_1^{3}x_2x_3^{3}x_4^{2}x_5^{14} $& 343. & $x_1^{3}x_2x_3^{3}x_4^{6}x_5^{10} $& 344. & $x_1^{3}x_2x_3^{3}x_4^{14}x_5^{2} $\cr  
345. & $x_1^{3}x_2x_3^{6}x_4^{2}x_5^{11} $& 346. & $x_1^{3}x_2x_3^{6}x_4^{3}x_5^{10} $& 347. & $x_1^{3}x_2x_3^{6}x_4^{10}x_5^{3} $\cr  
348. & $x_1^{3}x_2x_3^{6}x_4^{11}x_5^{2} $& 349. & $x_1^{3}x_2x_3^{7}x_4^{2}x_5^{10} $& 350. & $x_1^{3}x_2x_3^{7}x_4^{10}x_5^{2} $\cr  
351. & $x_1^{3}x_2x_3^{14}x_4^{2}x_5^{3} $& 352. & $x_1^{3}x_2x_3^{14}x_4^{3}x_5^{2} $& 353. & $x_1^{3}x_2x_3^{15}x_4^{2}x_5^{2} $\cr  
354. & $x_1^{3}x_2^{3}x_3x_4^{2}x_5^{14} $& 355. & $x_1^{3}x_2^{3}x_3x_4^{6}x_5^{10} $& 356. & $x_1^{3}x_2^{3}x_3x_4^{14}x_5^{2} $\cr  
357. & $x_1^{3}x_2^{3}x_3^{5}x_4^{2}x_5^{10} $& 358. & $x_1^{3}x_2^{3}x_3^{5}x_4^{10}x_5^{2} $& 359. & $x_1^{3}x_2^{3}x_3^{13}x_4^{2}x_5^{2} $\cr  
360. & $x_1^{3}x_2^{5}x_3^{2}x_4^{2}x_5^{11} $& 361. & $x_1^{3}x_2^{5}x_3^{2}x_4^{3}x_5^{10} $& 362. & $x_1^{3}x_2^{5}x_3^{2}x_4^{10}x_5^{3} $\cr  
363. & $x_1^{3}x_2^{5}x_3^{2}x_4^{11}x_5^{2} $& 364. & $x_1^{3}x_2^{5}x_3^{3}x_4^{2}x_5^{10} $& 365. & $x_1^{3}x_2^{5}x_3^{3}x_4^{10}x_5^{2} $\cr  
366. & $x_1^{3}x_2^{5}x_3^{10}x_4^{2}x_5^{3} $& 367. & $x_1^{3}x_2^{5}x_3^{10}x_4^{3}x_5^{2} $& 368. & $x_1^{3}x_2^{5}x_3^{11}x_4^{2}x_5^{2} $\cr  
369. & $x_1^{3}x_2^{7}x_3x_4^{2}x_5^{10} $& 370. & $x_1^{3}x_2^{7}x_3x_4^{10}x_5^{2} $& 371. & $x_1^{3}x_2^{7}x_3^{9}x_4^{2}x_5^{2} $\cr  
372. & $x_1^{3}x_2^{13}x_3^{2}x_4^{2}x_5^{3} $& 373. & $x_1^{3}x_2^{13}x_3^{2}x_4^{3}x_5^{2} $& 374. & $x_1^{3}x_2^{13}x_3^{3}x_4^{2}x_5^{2} $\cr  
375. & $x_1^{3}x_2^{15}x_3x_4^{2}x_5^{2} $& 376. & $x_1^{7}x_2x_3^{2}x_4^{2}x_5^{11} $& 377. & $x_1^{7}x_2x_3^{2}x_4^{3}x_5^{10} $\cr  
378. & $x_1^{7}x_2x_3^{2}x_4^{10}x_5^{3} $& 379. & $x_1^{7}x_2x_3^{2}x_4^{11}x_5^{2} $& 380. & $x_1^{7}x_2x_3^{3}x_4^{2}x_5^{10} $\cr  
381. & $x_1^{7}x_2x_3^{3}x_4^{10}x_5^{2} $& 382. & $x_1^{7}x_2x_3^{10}x_4^{2}x_5^{3} $& 383. & $x_1^{7}x_2x_3^{10}x_4^{3}x_5^{2} $\cr  
384. & $x_1^{7}x_2x_3^{11}x_4^{2}x_5^{2} $& 385. & $x_1^{7}x_2^{3}x_3x_4^{2}x_5^{10} $& 386. & $x_1^{7}x_2^{3}x_3x_4^{10}x_5^{2} $\cr  
387. & $x_1^{7}x_2^{3}x_3^{9}x_4^{2}x_5^{2} $& 388. & $x_1^{7}x_2^{9}x_3^{2}x_4^{2}x_5^{3} $& 389. & $x_1^{7}x_2^{9}x_3^{2}x_4^{3}x_5^{2} $\cr  
390. & $x_1^{7}x_2^{9}x_3^{3}x_4^{2}x_5^{2} $& 391. & $x_1^{7}x_2^{11}x_3x_4^{2}x_5^{2} $& 392. & $x_1^{15}x_2x_3^{2}x_4^{2}x_5^{3} $\cr  
393. & $x_1^{15}x_2x_3^{2}x_4^{3}x_5^{2} $& 394. & $x_1^{15}x_2x_3^{3}x_4^{2}x_5^{2} $& 395. & $x_1^{15}x_2^{3}x_3x_4^{2}x_5^{2} $\cr 
\end{tabular}}
\medskip
$B_5^+(3,4,3)$ is the set of 24 monomials $b_t = b_{23,t}, \ 396 \leqslant t \leqslant 419$:

\medskip
 \centerline{\begin{tabular}{rlrlrlrl}
396.&  $x_1x_2^{3}x_3^{6}x_4^{6}x_5^{7}$ &  397.&  $x_1x_2^{3}x_3^{6}x_4^{7}x_5^{6}$ &  398.&  $x_1x_2^{3}x_3^{7}x_4^{6}x_5^{6}$ &  399.&  $x_1x_2^{7}x_3^{3}x_4^{6}x_5^{6}$\cr  
400.&  $x_1^{3}x_2x_3^{6}x_4^{6}x_5^{7}$ &  401.&  $x_1^{3}x_2x_3^{6}x_4^{7}x_5^{6}$ &  402.&  $x_1^{3}x_2x_3^{7}x_4^{6}x_5^{6}$ &  403.&  $x_1^{3}x_2^{3}x_3^{5}x_4^{6}x_5^{6}$\cr  
404.&  $x_1^{3}x_2^{5}x_3^{2}x_4^{6}x_5^{7}$ &  405.&  $x_1^{3}x_2^{5}x_3^{2}x_4^{7}x_5^{6}$ &  406.&  $x_1^{3}x_2^{5}x_3^{3}x_4^{6}x_5^{6}$ &  407.&  $x_1^{3}x_2^{5}x_3^{6}x_4^{2}x_5^{7}$\cr  
408.&  $x_1^{3}x_2^{5}x_3^{6}x_4^{3}x_5^{6}$ &  409.&  $x_1^{3}x_2^{5}x_3^{6}x_4^{6}x_5^{3}$ &  410.&  $x_1^{3}x_2^{5}x_3^{6}x_4^{7}x_5^{2}$ &  411.&  $x_1^{3}x_2^{5}x_3^{7}x_4^{2}x_5^{6}$\cr  
412.&  $x_1^{3}x_2^{5}x_3^{7}x_4^{6}x_5^{2}$ &  413.&  $x_1^{3}x_2^{7}x_3x_4^{6}x_5^{6}$ &  414.&  $x_1^{3}x_2^{7}x_3^{5}x_4^{2}x_5^{6}$ &  415.&  $x_1^{3}x_2^{7}x_3^{5}x_4^{6}x_5^{2}$\cr  
416.&  $x_1^{7}x_2x_3^{3}x_4^{6}x_5^{6}$ &  417.&  $x_1^{7}x_2^{3}x_3x_4^{6}x_5^{6}$ &  418.&  $x_1^{7}x_2^{3}x_3^{5}x_4^{2}x_5^{6}$ &  419.&  $x_1^{7}x_2^{3}x_3^{5}x_4^{6}x_5^{2}$\cr   
\end{tabular}}

\medskip
\subsection{Some $\Sigma_5$-invariant classes of degree 23 in ${P_5}$.}\ \label{s74}

\medskip
We list here some polynomials which present the $\Sigma_5$-invariant classes of degree $23$ in $QP_5$.

\begin{align*} p_4 &= x_1x_2^{2}x_3^{3}x_4^{3}x_5^{14} + x_1x_2^{2}x_3^{3}x_4^{14}x_5^{3} + x_1x_2^{2}x_3^{7}x_4^{3}x_5^{10} + x_1x_2^{2}x_3^{7}x_4^{10}x_5^{3}\\ 
&\quad + x_1x_2^{3}x_3^{2}x_4^{3}x_5^{14} + x_1x_2^{3}x_3^{2}x_4^{14}x_5^{3} + x_1x_2^{3}x_3^{3}x_4^{2}x_5^{14} + x_1x_2^{3}x_3^{3}x_4^{14}x_5^{2}\\ 
&\quad + x_1x_2^{3}x_3^{14}x_4^{2}x_5^{3} + x_1x_2^{3}x_3^{14}x_4^{3}x_5^{2} + x_1x_2^{7}x_3^{2}x_4^{3}x_5^{10} + x_1x_2^{7}x_3^{2}x_4^{10}x_5^{3}\\ 
&\quad + x_1x_2^{7}x_3^{3}x_4^{2}x_5^{10} + x_1x_2^{7}x_3^{3}x_4^{10}x_5^{2} + x_1x_2^{7}x_3^{10}x_4^{2}x_5^{3} + x_1x_2^{7}x_3^{10}x_4^{3}x_5^{2}\\ 
&\quad + x_1^{3}x_2x_3^{2}x_4^{3}x_5^{14} + x_1^{3}x_2x_3^{2}x_4^{14}x_5^{3} + x_1^{3}x_2x_3^{3}x_4^{2}x_5^{14} + x_1^{3}x_2x_3^{3}x_4^{14}x_5^{2}\\ 
&\quad + x_1^{3}x_2x_3^{14}x_4^{2}x_5^{3} + x_1^{3}x_2x_3^{14}x_4^{3}x_5^{2} + x_1^{3}x_2^{3}x_3x_4^{2}x_5^{14} + x_1^{3}x_2^{3}x_3x_4^{14}x_5^{2}\\ 
&\quad + x_1^{3}x_2^{3}x_3^{13}x_4^{2}x_5^{2} + x_1^{3}x_2^{13}x_3^{2}x_4^{2}x_5^{3} + x_1^{3}x_2^{13}x_3^{2}x_4^{3}x_5^{2} + x_1^{3}x_2^{13}x_3^{3}x_4^{2}x_5^{2}\\ 
&\quad + x_1^{7}x_2x_3^{2}x_4^{3}x_5^{10} + x_1^{7}x_2x_3^{2}x_4^{10}x_5^{3} + x_1^{7}x_2x_3^{3}x_4^{2}x_5^{10} + x_1^{7}x_2x_3^{3}x_4^{10}x_5^{2}\\ 
&\quad + x_1^{7}x_2x_3^{10}x_4^{2}x_5^{3} + x_1^{7}x_2x_3^{10}x_4^{3}x_5^{2} + x_1^{7}x_2^{3}x_3x_4^{2}x_5^{10} + x_1^{7}x_2^{3}x_3x_4^{10}x_5^{2}\\ 
&\quad + x_1^{7}x_2^{3}x_3^{9}x_4^{2}x_5^{2} + x_1^{7}x_2^{9}x_3^{2}x_4^{2}x_5^{3} + x_1^{7}x_2^{9}x_3^{2}x_4^{3}x_5^{2} + x_1^{7}x_2^{9}x_3^{3}x_4^{2}x_5^{2}.
\end{align*}
\begin{align*}
p_5 &= x_2x_3x_4^{7}x_5^{14} +  x_2x_3^{7}x_4x_5^{14} +  x_2x_3^{7}x_4^{14}x_5 +  x_2^{7}x_3x_4x_5^{14} +  x_2^{7}x_3x_4^{14}x_5 \\ 
&\quad +  x_1x_3x_4^{7}x_5^{14} +  x_1x_3^{7}x_4x_5^{14} +  x_1x_3^{7}x_4^{14}x_5 +  x_1x_2x_4^{7}x_5^{14} +  x_1x_2x_3^{7}x_5^{14} \\ 
&\quad +  x_1x_2x_3^{7}x_4^{14} +  x_1x_2^{7}x_4x_5^{14} +  x_1x_2^{7}x_4^{14}x_5 +  x_1x_2^{7}x_3x_5^{14} +  x_1x_2^{7}x_3x_4^{14} \\ 
&\quad +  x_1x_2^{7}x_3^{14}x_5 +  x_1x_2^{7}x_3^{14}x_4 +  x_1^{7}x_3x_4x_5^{14} +  x_1^{7}x_3x_4^{14}x_5 +  x_1^{7}x_2x_4x_5^{14}\\ 
&\quad +  x_1^{7}x_2x_4^{14}x_5 +  x_1^{7}x_2x_3x_5^{14} +  x_1^{7}x_2x_3x_4^{14} +  x_1^{7}x_2x_3^{14}x_5 +  x_1^{7}x_2x_3^{14}x_4 \\ 
&\quad +  x_2x_3^{3}x_4^{13}x_5^{6} +  x_2^{3}x_3x_4^{13}x_5^{6} +  x_2^{3}x_3^{13}x_4x_5^{6} +  x_2^{3}x_3^{13}x_4^{6}x_5 +  x_1x_3^{3}x_4^{13}x_5^{6} \\ 
&\quad +  x_1x_2^{3}x_4^{13}x_5^{6} +  x_1x_2^{3}x_3^{13}x_5^{6} +  x_1x_2^{3}x_3^{13}x_4^{6} +  x_1^{3}x_3x_4^{13}x_5^{6} +  x_1^{3}x_3^{13}x_4x_5^{6} \\ 
&\quad +  x_1^{3}x_3^{13}x_4^{6}x_5 +  x_1^{3}x_2x_4^{13}x_5^{6} +  x_1^{3}x_2x_3^{13}x_5^{6} +  x_1^{3}x_2x_3^{13}x_4^{6} +  x_1^{3}x_2^{13}x_4x_5^{6} \\ 
&\quad +  x_1^{3}x_2^{13}x_4^{6}x_5 +  x_1^{3}x_2^{13}x_3x_5^{6} +  x_1^{3}x_2^{13}x_3x_4^{6} +  x_1^{3}x_2^{13}x_3^{6}x_5 +  x_1^{3}x_2^{13}x_3^{6}x_4 \\ 
&\quad +  x_2^{7}x_3^{11}x_4x_5^{4} +  x_2^{7}x_3^{11}x_4^{4}x_5 +  x_1^{7}x_3^{11}x_4x_5^{4} +  x_1^{7}x_3^{11}x_4^{4}x_5 +  x_1^{7}x_2^{11}x_4x_5^{4} \\ 
&\quad +  x_1^{7}x_2^{11}x_4^{4}x_5 +  x_1^{7}x_2^{11}x_3x_5^{4} +  x_1^{7}x_2^{11}x_3x_4^{4} +  x_1^{7}x_2^{11}x_3^{4}x_5 +  x_1^{7}x_2^{11}x_3^{4}x_4 \\ 
&\quad +  x_2x_3^{6}x_4^{7}x_5^{9} +  x_1x_3^{6}x_4^{7}x_5^{9} +  x_1x_2^{6}x_4^{7}x_5^{9} +  x_1x_2^{6}x_3^{7}x_5^{9} +  x_1x_2^{6}x_3^{7}x_4^{9} \\ 
&\quad +  x_2^{3}x_3^{4}x_4^{7}x_5^{9} +  x_1^{3}x_3^{4}x_4^{7}x_5^{9} +  x_1^{3}x_2^{4}x_4^{7}x_5^{9} +  x_1^{3}x_2^{4}x_3^{7}x_5^{9} +  x_1^{3}x_2^{4}x_3^{7}x_4^{9}.
\end{align*}

\begin{align*}
p_6 &= x_2x_3^{7}x_4^{7}x_5^{8} +  x_2^{7}x_3x_4^{7}x_5^{8} +  x_2^{7}x_3^{7}x_4x_5^{8} +  x_2^{7}x_3^{7}x_4^{8}x_5 +  x_1x_3^{7}x_4^{7}x_5^{8} \\ 
&\quad +  x_1x_2^{7}x_4^{7}x_5^{8} +  x_1x_2^{7}x_3^{7}x_5^{8} +  x_1x_2^{7}x_3^{7}x_4^{8} +  x_1^{7}x_3x_4^{7}x_5^{8} +  x_1^{7}x_3^{7}x_4x_5^{8} \\ 
&\quad +  x_1^{7}x_3^{7}x_4^{8}x_5 +  x_1^{7}x_2x_4^{7}x_5^{8} +  x_1^{7}x_2x_3^{7}x_5^{8} +  x_1^{7}x_2x_3^{7}x_4^{8} +  x_1^{7}x_2^{7}x_4x_5^{8} \\ 
&\quad +  x_1^{7}x_2^{7}x_4^{8}x_5 +  x_1^{7}x_2^{7}x_3x_5^{8} +  x_1^{7}x_2^{7}x_3x_4^{8} +  x_1^{7}x_2^{7}x_3^{8}x_5 +  x_1^{7}x_2^{7}x_3^{8}x_4 \\ 
&\quad +  x_2x_3^{3}x_4^{7}x_5^{12} +  x_2^{7}x_3^{3}x_4x_5^{12} +  x_2^{7}x_3^{3}x_4^{12}x_5 +  x_1x_3^{3}x_4^{7}x_5^{12} +  x_1x_2^{3}x_4^{7}x_5^{12} \\ 
&\quad +  x_1x_2^{3}x_3^{7}x_5^{12} +  x_1x_2^{3}x_3^{7}x_4^{12} +  x_1^{7}x_3^{3}x_4x_5^{12} +  x_1^{7}x_3^{3}x_4^{12}x_5 +  x_1^{7}x_2^{3}x_4x_5^{12}\\ 
&\quad +  x_1^{7}x_2^{3}x_4^{12}x_5 +  x_1^{7}x_2^{3}x_3x_5^{12} +  x_1^{7}x_2^{3}x_3x_4^{12} +  x_1^{7}x_2^{3}x_3^{12}x_5 +  x_1^{7}x_2^{3}x_3^{12}x_4 \\ 
&\quad +  x_2x_3^{7}x_4^{6}x_5^{9} +  x_2^{7}x_3x_4^{6}x_5^{9} +  x_1x_3^{7}x_4^{6}x_5^{9} +  x_1x_2^{7}x_4^{6}x_5^{9} +  x_1x_2^{7}x_3^{6}x_5^{9} \\ 
&\quad +  x_1x_2^{7}x_3^{6}x_4^{9} +  x_1^{7}x_3x_4^{6}x_5^{9} +  x_1^{7}x_2x_4^{6}x_5^{9} +  x_1^{7}x_2x_3^{6}x_5^{9} +  x_1^{7}x_2x_3^{6}x_4^{9} \\ 
&\quad +  x_2^{3}x_3^{7}x_4^{4}x_5^{9} +  x_2^{3}x_3^{7}x_4^{9}x_5^{4} +  x_2^{7}x_3^{3}x_4^{4}x_5^{9} +  x_2^{7}x_3^{3}x_4^{9}x_5^{4} +  x_1^{3}x_3^{7}x_4^{4}x_5^{9} \\ 
&\quad +  x_1^{3}x_3^{7}x_4^{9}x_5^{4} +  x_1^{3}x_2^{7}x_4^{4}x_5^{9} +  x_1^{3}x_2^{7}x_4^{9}x_5^{4} +  x_1^{3}x_2^{7}x_3^{4}x_5^{9} +  x_1^{3}x_2^{7}x_3^{4}x_4^{9} \\ 
&\quad +  x_1^{3}x_2^{7}x_3^{9}x_5^{4} +  x_1^{3}x_2^{7}x_3^{9}x_4^{4} +  x_1^{7}x_3^{3}x_4^{4}x_5^{9} +  x_1^{7}x_3^{3}x_4^{9}x_5^{4} +  x_1^{7}x_2^{3}x_4^{4}x_5^{9} \\ 
&\quad +  x_1^{7}x_2^{3}x_4^{9}x_5^{4} +  x_1^{7}x_2^{3}x_3^{4}x_5^{9} +  x_1^{7}x_2^{3}x_3^{4}x_4^{9} +  x_1^{7}x_2^{3}x_3^{9}x_5^{4} +  x_1^{7}x_2^{3}x_3^{9}x_4^{4}.
\end{align*}
\begin{align*}
p_7 &= x_2x_3^{3}x_4^{5}x_5^{14} +  x_2x_3^{3}x_4^{14}x_5^{5} +  x_2^{3}x_3x_4^{5}x_5^{14} +  x_2^{3}x_3^{5}x_4x_5^{14} +  x_2^{3}x_3^{5}x_4^{14}x_5 \\ 
&\quad +  x_1x_3^{3}x_4^{5}x_5^{14} +  x_1x_3^{3}x_4^{14}x_5^{5} +  x_1x_2^{3}x_4^{5}x_5^{14} +  x_1x_2^{3}x_4^{14}x_5^{5} +  x_1x_2^{3}x_3^{5}x_5^{14} \\ 
&\quad +  x_1x_2^{3}x_3^{5}x_4^{14} +  x_1x_2^{3}x_3^{14}x_5^{5} +  x_1x_2^{3}x_3^{14}x_4^{5} +  x_1^{3}x_3x_4^{5}x_5^{14} +  x_1^{3}x_3^{5}x_4x_5^{14} \\ 
&\quad +  x_1^{3}x_3^{5}x_4^{14}x_5 +  x_1^{3}x_2x_4^{5}x_5^{14} +  x_1^{3}x_2x_3^{5}x_5^{14} +  x_1^{3}x_2x_3^{5}x_4^{14} +  x_1^{3}x_2^{5}x_4x_5^{14} \\ 
&\quad +  x_1^{3}x_2^{5}x_4^{14}x_5 +  x_1^{3}x_2^{5}x_3x_5^{14} +  x_1^{3}x_2^{5}x_3x_4^{14} +  x_1^{3}x_2^{5}x_3^{14}x_5 +  x_1^{3}x_2^{5}x_3^{14}x_4 \\ 
&\quad +  x_2x_3^{3}x_4^{6}x_5^{13} +  x_2x_3^{6}x_4^{3}x_5^{13} +  x_1x_3^{3}x_4^{6}x_5^{13} +  x_1x_3^{6}x_4^{3}x_5^{13} +  x_1x_2^{3}x_4^{6}x_5^{13} \\ 
&\quad +  x_1x_2^{3}x_3^{6}x_5^{13} +  x_1x_2^{3}x_3^{6}x_4^{13} +  x_1x_2^{6}x_4^{3}x_5^{13} +  x_1x_2^{6}x_3^{3}x_5^{13} +  x_1x_2^{6}x_3^{3}x_4^{13} \\ 
&\quad +  x_2^{3}x_3x_4^{7}x_5^{12} +  x_2^{3}x_3^{7}x_4x_5^{12} +  x_2^{3}x_3^{7}x_4^{12}x_5 +  x_1^{3}x_3x_4^{7}x_5^{12} +  x_1^{3}x_3^{7}x_4x_5^{12} \\ 
&\quad +  x_1^{3}x_3^{7}x_4^{12}x_5 +  x_1^{3}x_2x_4^{7}x_5^{12} +  x_1^{3}x_2x_3^{7}x_5^{12} +  x_1^{3}x_2x_3^{7}x_4^{12} +  x_1^{3}x_2^{7}x_4x_5^{12} \\ 
&\quad +  x_1^{3}x_2^{7}x_4^{12}x_5 +  x_1^{3}x_2^{7}x_3x_5^{12} +  x_1^{3}x_2^{7}x_3x_4^{12} +  x_1^{3}x_2^{7}x_3^{12}x_5 +  x_1^{3}x_2^{7}x_3^{12}x_4 \\ 
&\quad +  x_2x_3^{6}x_4^{11}x_5^{5} +  x_1x_3^{6}x_4^{11}x_5^{5} +  x_1x_2^{6}x_4^{11}x_5^{5} +  x_1x_2^{6}x_3^{11}x_5^{5} +  x_1x_2^{6}x_3^{11}x_4^{5} \\ 
&\quad +  x_2^{3}x_3^{3}x_4^{4}x_5^{13} +  x_2^{3}x_3^{3}x_4^{13}x_5^{4} +  x_2^{3}x_3^{4}x_4^{3}x_5^{13} +  x_1^{3}x_3^{3}x_4^{4}x_5^{13} +  x_1^{3}x_3^{3}x_4^{13}x_5^{4} \\ 
&\quad +  x_1^{3}x_3^{4}x_4^{3}x_5^{13} +  x_1^{3}x_2^{3}x_4^{4}x_5^{13} +  x_1^{3}x_2^{3}x_4^{13}x_5^{4} +  x_1^{3}x_2^{3}x_3^{4}x_5^{13} +  x_1^{3}x_2^{3}x_3^{4}x_4^{13} \\ 
&\quad +  x_1^{3}x_2^{3}x_3^{13}x_5^{4} +  x_1^{3}x_2^{3}x_3^{13}x_4^{4} +  x_1^{3}x_2^{4}x_4^{3}x_5^{13} +  x_1^{3}x_2^{4}x_3^{3}x_5^{13} +  x_1^{3}x_2^{4}x_3^{3}x_4^{13} \\ 
&\quad +  x_2^{3}x_3^{4}x_4^{11}x_5^{5} +  x_1^{3}x_3^{4}x_4^{11}x_5^{5} +  x_1^{3}x_2^{4}x_4^{11}x_5^{5} +  x_1^{3}x_2^{4}x_3^{11}x_5^{5} +  x_1^{3}x_2^{4}x_3^{11}x_4^{5} \\ 
&\quad + x_2^{3}x_3^{3}x_4^{5}x_5^{12} +  x_2^{3}x_3^{3}x_4^{12}x_5^{5} +  x_1^{3}x_3^{3}x_4^{5}x_5^{12} +  x_1^{3}x_3^{3}x_4^{12}x_5^{5} +  x_1^{3}x_2^{3}x_4^{5}x_5^{12} \\ 
&\quad +  x_1^{3}x_2^{3}x_4^{12}x_5^{5} +  x_1^{3}x_2^{3}x_3^{5}x_5^{12} +  x_1^{3}x_2^{3}x_3^{5}x_4^{12} +  x_1^{3}x_2^{3}x_3^{12}x_5^{5} +  x_1^{3}x_2^{3}x_3^{12}x_4^{5}.
\end{align*}

\medskip\noindent
{\bf Acknowledgment.} I would like to express my warmest thanks to my adviser, Asso. Prof. Nguyen Sum, for his inspiring guidance and generous help in finding the proofs as good as in describing the results.

I would like to thank the University of Technology and Education H\`{\^o} Ch\'i Minh city for supporting this work.

{}

\vskip0.2cm

Faculty of Foundation Sciences, 

University of Technology and Education H\`\ocircumflex\  Ch\'i Minh city,

01 V\~o V\u an Ng\^an, Th\h u \DJ\'\uhorn c district, H\`\ocircumflex\ Ch\'i Minh city, Viet Nam.

Email:  tinnk@hcmute.edu.vn

\end{document}